\documentclass[10pt]{article}
    \usepackage[utf8]{inputenc}
    \newcommand{\aff}{\operatorname{aff}}
    \usepackage{algorithm}
    \usepackage{authblk}
     \usepackage{algpseudocode}
    \usepackage{fullpage}
    \usepackage{natbib}
    \usepackage{amsthm}
    \usepackage{color}
    \usepackage{amsmath}
    \usepackage{amsfonts}
    \usepackage{xcolor}
    \colorlet{linkequation}{blue}
    \usepackage{framed}
    \usepackage{times}
    \usepackage{amsmath}
    \usepackage{amsthm}
    \usepackage{amssymb}
    \usepackage{wasysym}
    \usepackage{mathrsfs}
    \usepackage{bbm}
    \usepackage{geometry} 
    \usepackage{mathtools}
    \usepackage[colorlinks=true,citecolor=blue,linkcolor=blue]{hyperref}
        \usepackage{cleveref}

    \newtheorem{theorem}{Theorem}

    \newtheorem{definition}[theorem]{Definition}
    
    \newtheorem{lemma}[theorem]{Lemma}
    \newtheorem{corollary}[theorem]{Corollary}
    \newtheorem{remark}[theorem]{Remark}

    \newtheorem{assumption}[theorem]{Assumption}

    \newcommand{\E}{{\mathbb E}}
    
    \newcommand{\R}{{\mathbb R}}

    \newcommand{\PP}{\mathbb{P}}

    \newcommand{\eps}{\epsilon}

    \newcommand{\conv}{\mathrm{conv}}
    \newcounter{rcnt}[section]

    \newcommand{\asyd}{\underset{d}{\asymp}}
    \def\argmin{\mathop{\rm argmin}}

    \newcommand{\vol}{\operatorname{vol}}

    \newcommand{\QQ}{\mathbb Q}
    
    \newcommand{\cA}{{\mathcal A}}
    \newcommand{\cB}{{\mathcal B}}
    
    \newcommand{\cD}{{\mathcal D}}
    
    \newcommand{\cF}{{\mathcal F}}

    \newcommand{\cM}{{\mathcal M}}
    \newcommand{\cN}{{\mathcal N}}
    \newcommand{\cO}{{\mathcal O}}

    \newcommand{\cR}{{\mathcal R}}
    \newcommand{\cS}{{\mathcal S}}

    \newcommand{\lesssimd}{\underset{d}{\lesssim}}
    \newcommand{\gtrsimd}{\underset{d}{\gtrsim}}

    \usepackage{lipsum}
    \newcommand{\erm}{\widehat{f}_n}
    
    \usepackage[T1]{fontenc}
    \usepackage{authblk}
    \author[1]{Gil Kur} 
    \author[2]{Eli  Putterman}
    \affil[1]{Electrical Engineering and Computer Science, Massachusetts Institute of Technology }
    \affil[2]{Department of Mathematics, Tel Aviv University}

    \title{An Efficient Minimax Optimal Estimator For Multivariate Convex Regression}
    
\begin{document}
    \date{}
    \maketitle
\begin{abstract}%
This work studies the computational aspects of multivariate convex regression in dimensions $d \ge 5$. Our results include the \emph{first} estimators that are minimax optimal (up to logarithmic factors) with polynomial runtime in the sample size for both $L$-Lipschitz convex regression, and $\Gamma$-bounded convex regression under polytopal support. Our analysis combines techniques from empirical process theory, stochastic geometry, and potential theory, and leverages recent algorithmic advances in mean estimation for random vectors and in distribution-free linear regression. These results provide the first efficient, minimax-optimal procedures for non-Donsker classes for which their corresponding least-squares estimator is provably minimax-suboptimal.

     %In the remaining case of $d \leq 4$, the corresponding Least Squares Estimators  for these tasks are efficient minimax optimal estimators.
    %  We also provide a new approach to construct computationally efficient minimax optimal estimators in the task of non-parametric regression, a result of independent interest.
\end{abstract}

    \section{Introduction and Main Results}
    %   \blfootnote{Keywords: Multivariate Convex Regression; Efficient Minimax Optimal Estimator; Non-Donsker Classes;}%
    
% \begin{keywords}%
%   Multivariate Convex Regression; Minimax Optimality; Non-Donsker Regime  %
% \end{keywords}
 We consider the task of regression in the well-specified  model in the random design setting: 
    \[
        Y = f^{*}(X) + \xi
    \]
    where $f^{*}:\Omega \to \R$ lies in a known function class $\cF$, $X$ is drawn from a \emph{known} distribution $\PP$ on $\Omega$, and $\xi$ is a zero-mean noise with a finite variance $\sigma^2$.
    
   Given $(\cF,\PP)$ and $n$ i.i.d. observations $\cD:=\{(X_1,Y_1),\ldots,(X_n,Y_n)\}$,  we aim to estimate the underlying function $f^{*}$ as well as possible with respect to the classical minimax risk  \citep{tysbakovnon}. More precisely, the risk of  an estimator $\bar{f}:\cD \to \cB(\Omega)$, where $\cB(\Omega)$ are the measurable functions, is defined as
    \[
        \cR_n(\bar{f},\cF,\PP):= \sup_{f^{*} \in \cF} \E_{\cD} \int (f^{*} - \bar{f})^2d\PP,
    \]
    and the minimax rate equals to \[
    \cM_n(\cF,\PP):= \inf_{\bar{f}}\cR_n(\bar{f},\cF,\PP).\] 
    Our goal is to find an estimator that is \emph{minimax optimal} up to logarithmic factors in $n$, i.e. its risk satisfies 
    \[
        \cR_n(\bar{f},\cF,\PP) = \Tilde{O}_{\PP,\cF}(\cM_n(\cF,\PP)):= \cM_n(\cF,\PP) \cdot \log(n)^{O_{\PP,\cF}(1)},
    \]
    where $O_{\PP,\cF}(\cdot)$ denotes equality up to a multiplicative constant that depends only on $\PP, \cF$. We would also like our estimator to be  \emph{efficiently computable}, which for our purposes means that its runtime is polynomial in the number of samples: i.e. of order $n^{O_{\PP,\cF}(1)}$ when $n$ is large enough.  In this paper, $B_d$ denotes the unit (Euclidean) ball in dimension $d$, $\cF$ is one of the following two subclasses of convex functions
    % recall that the minimax rate  considers the inefficient estimators as well in the number of samples, i.e. of at most $O_d(n^{O_d(1)})$,
    % recall that the minimax rate  considers the inefficient estimators as well.  
    \begin{enumerate}
        \item $\cF_L(\Omega)$ --- the class of convex $L$-Lipschitz functions supported on a convex set $\Omega \subseteq B_d$.
        % $f^{*} \in \cF_{L}(\Omega)$ Estimating a convex Lipschitz function that is supported on a convex domain $\Omega \subset \R^{d}$.
        \item $\cF^\Gamma(P)$ --- the class of convex functions on a polytope $P \subset B_d$ unifromly bouned by $[-\Gamma, \Gamma]$.
    \end{enumerate}
    These tasks are known as $L$-Lipschitz convex regression  \citep{seijo2011nonparametric} and $\Gamma$-bounded convex regression  \citep{han2016multivariate}, respectively. For reasons which will become apparent later, we always take $d \ge 5$. We also further assume that $\PP$ satisfies the following:
    
    %Since, the definition of a convex function clearly implies a convex support, we prove our results under the assumption of $\PP$ being the uniform distribution over the domain $\Omega$ or $P$, see Remark \ref{Rem:Conference} below for additional details.  
    \begin{assumption}\label{Ass:One}
    $\PP$ is uniformly bounded on its support by some positive constants $c(d),C(d)\geq 0 $ that only depend on $d$, i.e. $c(d) \leq \frac{d\PP}{dx}(x) \leq C(d)$, for all $x \in \mathrm{Supp}(\PP)$.
    \end{assumption}
    % We remark that in the $\Gamma$-bounded regression, the minimax rates are different when the domain is a polytope with restricted number of facets or a domain is smooth convex body
    % and $n$ samples, $Y_1,\ldots,Y_n$,  of the underlying function $f^{*} \in \cF$, where $\cF$ is some known function class, that are corrupted by zero mean variance $\sigma^2$ guassian noise, i,e. 
    % \[
    % \forall \  1\leq i \leq n \quad  Y_i = f^{*}(X_i) + \xI_j \text{ where $\xI_j \underset{i.i.d}{\sim} N(0,\sigma^2) $,} 
    % \]
    % where in the first / second task  / $f^{*} \in \cF^{\Gamma}(\Omega)$ respectively. 
    
    % λ f0(x1) + (1−λ)f0(x2) ≥ f0(λx1 + (1−λ)x2),
    % for every x1,x2 ∈ X and λ ∈ (0,1). Given the observations (x1, y1),...,(xn, yn), we would like to
    % estimate f0 subject to the convexity constraint. Convex regression is easily extended to concave
    % regression since a concave function is the negative of a convex function.
    Convex regression tasks have been a central concern in the ``shape-constrained" statistics literature  \citep{devroye2012combinatorial}, and have innumerable applications in a variety of disciplines, from economic theory \citep{varian1982nonparametric} to operations research \citep{powell2003stochastic}  and more \citep{balazs2016convex}. In general, convexity is extensively studied in pure mathematics \citep{artstein2015asymptotic}, computer science \citep{lovasz2007geometry}, and optimization \citep{boyd2004convex}. We remark that there is a density-estimation counterpart of the convex regression problem, known as log-concave density estimation \citep{samworth2018recent,cule2010maximum}, and these two tasks are closely related \citep{kur2019optimality,kim2016global}.
    
    Due to the appearance of convex regression in various fields, it has been studied from many perspectives and by many different communities. For example, in the mathematical statistics literature the minimax rates of convex regression tasks and the risk of the Maximum Likelihood Estimator (MLE) are the main areas of interest; an incomplete sample of works treating this problem is \citep{guntuboyina2012optimal,guntuboyina2013covering,gardner1995geometric,gao2017entropy,Kur20,han2019global,brunel2013adaptive,diakonikolas2018polynomial,diakonikolas2016efficient,carpenter2018near,kur2019optimality,balazs2015near}. In operations research, work has focused on the algorithmic aspects of convex regression, i.e., finding scalable and efficient algorithms; see, e.g., \citep{ghosh2021max,brunel2016adaptive,o2021spectrahedral,soh2021fitting,balazs2016convex,mazumder2019computational,chen2020multivariate,bertsimas2021sparse,siahkamari2021faster,simchowitz2018adaptive,hannah2012ensemble,blanchet2019multivariate,lin2020efficient,chen2021newtwo,balazs2022adaptively}. Initially, convex regression was mostly studied in the univariate case, which is now considered to be well-understood. Multivariate convex regression has only begun to be explored in recent years, and is still an area of active research.
    
    % , whereas in the recent years multivariate .    
    % The task of convex regression has a very rich literature: The statistical aspects of this prolems such as the  
    % \subsection{The Task of Convex Regression}
    %  see an incomplete list   \citep{ghosh2021max,brunel2013adaptive,brunel2016adaptive,o2021spectrahedral,soh2021fitting,balazs2016convex,mazumder2019computational,chen2020multivariate,bertsimas2021sparse,siahkamari2021faster,hannah2013multivariate})
    % Our rest, we aim to find an estimator $\bar{f}$ that recover the underlying function $f^{*}$ in a \emph{minimax-optimal} way based on the observations $\{(X_i,Y_i)\}_{i=1}^{n}$ in a computationally efficient way, by that we mean that  that runs in polynomial time in the number samples, i.e. its runtime is $O_d(n^{O_d(1)})$ and also
    % \[
    %     \cA:\{ () \}
    % \] 
    % *************
    % From a statistical point of view, t
    
    The na\"ive algorithm for any variant of the convex regression task is the least squares estimator (LSE), which is also the MLE under Gaussian noise, defined by
    \begin{equation}\label{Eq:LSE}
\erm:= \argmin_{f \in \cF}\sum_{i=1}^{n}(Y_i-f(X_i))^2,
    \end{equation}
    where  $\cF = \cF_L(\Omega)$ or $\cF  = \cF^{\Gamma}(P)$ in our convex regression tasks.
    % , in the first or the second task respectively that appear. 
    From a computational point of view, the LSE can be formulated as a quadratic programming problem with $O(n^2)$ constraints, and is thus efficiently computable in our terms \citep{seijo2011nonparametric,han2016multivariate}; however, the LSE has been seen empirically not to be scalable for large number of samples \citep{chen2020multivariate}. 
    
    From a statistical point of view, the minimax rates of both of our convex regression tasks are $\Theta_{d,L,\sigma}(n^{-\frac{4}{d+4}})$ and $\Theta_{d,\Gamma,\sigma,}(C(P)\cdot n^{-\frac{4}{d+4}})$ for all $d \geq 1$ \citep{gao2017entropy,bronshtein1976varepsilon,yang1999information}. The LSE is minimax optimal \emph{only} in low dimension, when $d \leq 4$ \citep{birge1993rates}, while for $d \geq 5$ it attains a suboptimal risk of $\Tilde\Theta_d(n^{-\frac{2}{d}})$  \citep{Kur20,KurConv}. The poor statistical performance of the LSE for $d \ge 5$ has also been verified empirically   \citep{gardner2006convergence,ghosh2021max}. There are known minimax optimal estimators when $d \geq 5$, yet all of them are computationally inefficient. Moreover, all of them are based on some sort of discretization of the relevant function classes, i.e., they consider some $\eps$-nets (see Definition \ref{Def:Net} below). In our tasks, these algorithms require examining nets of cardinality of order $\exp(n^{\Theta(1)})$, and are thus perforce inefficient \citep{rakhlin2017empirical,guntuboyina2012optimal}. 
    
    % , we remark that when $d \geq 5$, these classes are known to be as non Donsker classes (see Def. below), and also sufferers from has non-parametric adaptation rates as was observed in ***.  
    % r our multivariate convex regression tasks (when $d \geq 5$)
    % the task of multivariate convex regression
    The empirically-observed poor performance of the LSE and the computational intractability of known minimax optimal estimators have motivated the study of efficient algorithms for convex regression with better statistical properties than the LSE; an incomplete list of relevant works appears above. However, previously studied algorithms are either provably minimax suboptimal or do not provide any statistical guarantees at all with respect to the minimax risk. We would however like to mention the ``adaptive partitioning'' estimator constructed in \cite{hannah2013multivariate}, which is the first provable computationally efficient estimator for convex regression which has been shown to be \emph{consistent} in the $L_{\infty}$ norm. The authors' approach is somewhat related to our proposed algorithm, but it is unknown whether their algorithm is minimax optimal.
    
    %Clearly,   we can ask the following question: Does exist a computationally efficient minimax optimal algorithm for the task of multivariate convex regression? Any answer has insightful implications: a negative one would essentially imply that minimax optimality may not be reachable in polynomial time, and therefore \emph{may not} be a practical measure of statistical performance. Whereas, an affirmative answer shows that minimax optimality can be reached in an efficient in the task of convex regression. Also, such an algorithm may provide some insights for improving known estimators that are used in practice for these tasks. Moreover, from a theoretical statistics point of view, it would imply that there are minimax optimal efficient estimators for non-Donsker classes (see Definition \ref{Def:Donsker} below), when the LSE is provably minimax suboptimal, a result that was not known before. 
    
    Our main results are the existence of computationally efficient minimax-optimal estimators for multivariate Lipschitz convex regression under \textit{any convex support} and bounded convex regression under \textit{polytopal support}.  To present the results, we use expressions such as $g(n) \le O_d(f(n))$ and $g(n) \lesssimd f(n)$ to mean that there exists $C_d \geq 0$ such that $g(n) \le C_d f(n)$ for all $n$. Specifically, we prove the following results:
    
    % \end{}
        % \center{ }
        % \item There exists a universal constant $C_1 \geq 0 $ such that $ \E [\xi^4] \leq C_1\sigma^4$;
    % \end{enumerate}
    %  (ii)  we prove the following results:
    \begin{theorem}
    \label{Thm:ConvexFunc}
    Let $d \geq 5$ and $n \geq d+1$. Then, under Assumption \ref{Ass:One}, in the task of  $L$-Lipschitz convex regression on a convex domain $\Omega  \subset B_d$, there exists an estimator, $\widehat{f}_{L}$, with runtime of at most $n^{O(d)}$ such that
    \begin{equation}
    \cR_{n}(\widehat{f}_{L},\cF_L(\Omega),\PP) \lesssimd (\sigma+L)^2n^{-\frac{4}{d+4}}\log(n)^{h(d)} 
    % O_d(\cM_{n}(\cF_{L}(\Omega),\PP))
    \end{equation}
     where $h(d) \leq 3d$.
    \end{theorem}
    % \begin{remark}
    % Theorem \ref{Thm:ConvexFunc} gives a minimax optimal estimator in many natural cases; e.g., it applies when the polytope $P$ is assumed to have only $C(d)$ vertices or facets, where $C(d)$ is a constant that only depends on $d$; this class includes, for instance, the unit cube, the simplex, and the $\ell_1$ ball. Then by \citep{mcmullen1970}, the second term in our bound is of order $\tilde{O}_d(n^{-\frac{d}{d+4}})$, which is strictly smaller than the minimax rate $\Theta_d(n^{-\frac{4}{d+4}})$.
    
% \begin{remark}
%          One can remove the redundant term that depends on $P$. Specifically, the following holds:
%          \[
%                 \cR_{n}(\widehat{f}_{L},\cF_{L}(\Omega),\PP) \leq O_d((\sigma+L)^2n^{-\frac{4}{d+4}}\log(n)^{h(d)})
%          \]
%          for any convex domain $\Omega \subset B_d$. The proof for an arbitrary convex body $\Omega$ involves a more involved and technical detour through stochastic geometry, which does not provide an additional insight.  
% \end{remark}
    
    \begin{theorem}\label{Thm:ConvexFunc2}
    Let $d \geq 5$ and $n \geq d+1$. Then,  under Assumption \ref{Ass:One}, for the task of  $\Gamma$-bounded convex regression on the polytope $ P \subset B_d$,  there exists an efficient estimator, $\widehat{f}^{\Gamma}$, with runtime of at most $n^{O(d)}$ such that
    \[
    \cR_{n}(\widehat{f}^{\Gamma},\cF^{\Gamma}(P),\PP) \lesssimd  C_1(P)(\sigma+\Gamma)^2n^{-\frac{4}{d+4}} \log(n)^{h(d)},
    %  O_d(\cM_{n}(\cF^{\Gamma}(P),\PP)) = \max\{(\sigma+L)^2,(\sigma+L)^4\}
    \]
    where $h(d) \leq 3d$ and  $C_1(P)$ is a constant that only depends on $P$.
    \end{theorem}
    
    As we mentioned earlier, for both of these two tasks the minimax rate is of order $n^{-\frac{4}{d+4}}$, so up to polylogarithmic factors in $n$, the above estimators are minimax optimal. We note that in Theorem \ref{Thm:ConvexFunc2}, the dependence of the constants on the polytope $P$ is unavoidable. This follows from the results of \citep{gao2017entropy,han2016multivariate}, in which the authors showed the geometry of the support of the measure $\PP$ affects the minimax rate of bounded convex regression. For example, in the extreme case $\Omega = B_d$, the minimax rate is of order $n^{-\frac{2}{d+1}}$, which is asymptotically larger than the error rate for polytopes; thus, if we take a sequence of polytopes which approaches $B_d$, the sequence of constants $C(P_n)$ will necessarily blow up.

\begin{remark}
        The dependence on the domain $P$ in Theorem \ref{Thm:ConvexFunc2} is actually an artifact of the use of the $L_2(\PP)$ metric, as the $L_2(\PP)$ entropy numbers of the class of $\Gamma$-bounded convex functions on a domain $\Omega$ depend strongly on the geometry of $\Omega$. One can show that for any convex domain $\Omega$ the minimax rate of bounded convex regression in the $L_1(\PP)$-metric is of order $\Theta_{d,\Gamma,\sigma}(n^{-\frac{2}{d+4}})$ (this is due to the fact that unlike the $L_2$-entropy numbers, the $L_1(\PP)$-entropy numbers do not depend on the domain), and there exists an efficient minimax optimal estimator attaining this rate. Thus, in the $L_1(\PP)$ setting, the minimax rate for convex regression is \emph{universal}. 
\end{remark}
    
    We consider our results as mainly a proof-of-concept for the existence of efficient estimators for the task of convex regression when $d \geq 5$. Due to their high polynomial runtime, in practice our estimators would probably not work well. However, as we mentioned above, the other minimax optimal estimators in the literature are computationally inefficient, and they all require consideration of some net of exponential size in $n$; our estimator is conceptually quite different. We hope that insights from our algorithm can be used to construct a practical estimator with the same desirable statistical properties.  From a purely theoretical point of view, our estimators are the first known minimax optimal efficient estimators for non-Donsker classes  for which their LSE are provably minimax suboptimal in $L_2$ (see Definition \ref{Def:Donsker} and Remark \ref{Rem:Donsker} below). Prior to this work, there were efficient optimal estimators for non-Donsker classes such that their corresponding LSE (or MLE) is provably efficient and optimal (such as log-concave density estimation and isotonic regression,  cf. \cite{kur2019optimality,han2019isotonic,han2019global,pananjady2022isotonic}).
    Our work should be contrasted with these earlier works. We show that it is possible to overcome the suboptimality of the LSE with an efficient optimal algorithm in the non-Donsker regime --- a result that was unknown before this work.

    % In Appendix \ref{Sec:Disc}, we show that our algorithm can be viewed as an instance of a new generic approach to non-parametric regression, which might be extensible to other non-parametric regression tasks.
    
    % The paper is organized as follows: In The next section, we present our estimator and a proof of its correctness. We conclude this introduction with a few remarks:
    %  In Appendix \ref{Sec:Disc}, we discuss our results through the lens of non-Donsker classes, and present the new approach which led us to constructing our efficient estimator.
    We prove Theorem \ref{Thm:ConvexFunc} in Section \ref{Sec:Proof}. The proof of Theorem \ref{Thm:ConvexFunc2} uses the same method as that of Theorem \ref{Thm:ConvexFunc}, along with the main result of \citep[Thm 1.1]{gao2017entropy}. We sketch the requisite modifications to the proof in Section \ref{Sec:Cor}. We conclude this introduction with the following remarks:
    \begin{remark}
    
    \begin{enumerate}
        \item We conjecture that the estimators of Theorems \ref{Thm:ConvexFunc} and \ref{Thm:ConvexFunc2} are minimax-optimal up to constants that only depend on $d,\sigma$, i.e the. $\log(n)$ factors are unnecessary.
        
        \item Our estimators' runtime is of order $n^{O(d)}$, which is much worse than the $O_d(n^{O(1)})$ runtime of the suboptimal convex LSE.   
         
        % The more challenging question is whether the assumptions of known $\PP$\footnote{In this case to be able to draw points uniformly from $\Omega$} and $\sigma^2$ can be relaxed. We leave these questions for future work.
        \item \label{Rem:affine} When $d \geq 5$, one can show that when $f^{*}$ is a max $k$-affine function (restricted to $P \subset B_d$), i.e. $f^{*}\mathbbm{1}_{P}(x) = \max_{1 \leq i \leq k}a_{i}^{\top}x + b_i$,
    our estimator attains a parametric rate, i.e.  
    \[
        \E \int (\widehat{f}^{\Gamma} - f^{*})^2d\PP \leq \Tilde{O}_d\left(\frac{C(P,k)}{n}\right).
    \]
    When $d \leq 4$, \cite{han2016multivariate} showed that  $\Gamma$-bounded convex LSE, that is defined in  \eqref{Eq:LSE} with $\cF = \cF^{\Gamma}(P)$, attains a parametric rate as well. However, when $d \geq 5$, the LSE attains a non-parametric error of $\Tilde{\Theta}_d(C(P,k)n^{-4/d})$ (\cite{KurConv}; for a more general result see \cite{Kur21}). Therefore, our algorithm has the proper adaptive rates when $d \geq 5$; see \cite{ghosh2021max} for more details.
    \item 
    \label{Rem:PR}
    An interesting property of our estimator is that the random design setting, i.e. the fact that data points $X_1,\ldots,X_n$ are drawn from $\PP$ rather than fixed, is essential to its success, a phenomenon not often observed when studying shape-constrained estimators. Usually these estimators also perform well on a ``nice enough" fixed design set, for example when $\Omega = [-1/2,1/2]^{d}$ and $X_1,\ldots,X_n$ are the regular grid points. 
    \end{enumerate}
    \end{remark}
    
    % \begin{remark}
    % In Corollary \ref{Cor:ConvexFunc}, the risk depends on the polytope $C(P)$ and this constant is essential \citep{han2016approximation}. Yet, if the risk were with respect to $L_1(\PP)$ (rather than the squared error), then one can show that the risk does not depend on $P$ anymore. Specifically, for any convex domain $\Omega \subset \R^d$, the following holds:
    % \[
    %  \sup_{f^{*} \in \cF^{\Gamma}(\Omega)}\E|\widehat{f}^{\Gamma}-f^{*}| \leq  O_d( \Gamma n^{-\frac{2}{d+4}}\log(n)^{\frac{d}{2}}).
    % \]
    % \end{remark}

    % We conclude this part of the introduction by a question that is the main motivation of this paper. and $\QQ_n'$ , and $\QQ_n' = n^{-1}\sum_{i=1}^{n}\delta_{X_{n+i}}$.

    \section{Proposed Estimator of Theorem \ref{Thm:ConvexFunc}}\label{Sec:Proof}
    % In this section, we present the estimator of Theorem \ref{Thm:ConvexFunc} and the proof of its correctness.
    % \begin{remark}\mbox{}\label{Rem:Conference}
    %  For completeness, we provide the full algorithm in Section \ref{SubSec:Modi} below.
        % \item It would seem natural to weaken our assumptions on $\PP$, and assume only that it is uniformly bounded on its support by some fixed positive constants (i.e., $c \leq \frac{d\PP}{dx}(x) \leq C$, for all $x \in \mathrm{supp}(\PP)$). Unfortunately, one of our crucial arguments (Lemma \ref{Lem:VerMSE} below) uses ideas from potential theory that only work when $\PP$ is uniform on the domain.
    % \end{remark}
    % For simplicity of notation and readability, we prove that our algort
    \subsection{Notations and preliminaries}
    Throughout this text, $C,C_1,C_2 \in (1,\infty)$ and $c,c_1,c_2,\ldots \in (0,1)$ are positive absolute constants that may change from line to line. Similarly, $ C(d),C_1(d),C_2(d),\ldots \in (1,\infty)$ and\\ $c(d),c_1(d),c_2(d),\ldots \in (0,1)$ are positive constants that only depend on $d$ that may change from line to line.  
    
    For any probability measure $\QQ$ and $m \geq 0$, we introduce the notation $\QQ_m$ for the random empirical measure of  $Z_{1},\ldots,Z_{m} \underset{i.i.d.}{\sim}\QQ$ , i.e. $\QQ_m = m^{-1}\sum_{i=1}^{m}\delta_{Z_i}$. Also, given a subset $A \subset \Omega$ of positive measure, we let $\mathbb P_A$ denote the conditional probability measure on $A$. For a positive integer $k$, $[k]$ denotes $\{1, \ldots, k\}$.
    
    % Also, we use the notation $\cD := \{(X_i,Y_i)\}_{i=1}^{n}$,
    % Also, for a given measurable set  denote by $int(A)$ to be the interior of the set $A$.
    \begin{definition}
    A simplex in $\R^{d}$ is the convex hull of $d+1$ points $v_1,\ldots,v_{d+1} \in \mathbb R^d$ which do not all lie in any hyperplane.
    \end{definition}
    % The following is a classical definition in stochastic geometry \cite{schutt1990convex}.
    % \begin{definition}
    % For every convex set $\Omega \subset \R^{d}$ and $\eps \in (0,1)$, we define its convexified floating body w.r.t. $\PP$
    % \[
    %     \Omega^{f}_{\eps}:= \bigcap_{\PP(\Omega \setminus H^{+}) \leq \eps} \PP(\Omega \cap H^{-}),
    % \]
    % where $H$ is a $(d-1)$-hyperplane. 
    % \end{definition}
    \begin{definition}\label{Def:ConvSimFunc}
    A convex function $f : \Omega \to \R$ is defined to be  $k$-simplicial if there exists $\triangle_1,\ldots,\triangle_k \subset \R^{\dim(\Omega)}$ simplices such that $\Omega = \bigcup_{i=1}^{k}\triangle_i$ and for each $1\leq i \leq k$, we have that   $f:\triangle_i \to \R$ is affine.
    \end{definition}
    Note that the definition is more restrictive than the usual definition of a $k$-max affine function (see Remark \ref{Rem:affine}), since the affine pieces of a $k$-max affine function are not constrained to be simplices.
    
    The following result from empirical process theory is a corollary of the peeling device \citep[Ch. 5]{van2000empirical}, \citep{bousquet2002concentration} and Bronshtein's entropy bound \citep{bronshtein1976varepsilon}.
    \begin{lemma}\label{Lem:Bro}
    Let $d \geq 5$, $m \geq C^{d}$ and $\QQ$ be a probability measure on $\Omega' \subset B_d$. Suppose $Z_1,\ldots,Z_m$ are drawn independently from $\QQ$; then with probability at least $1 - C_1(d) \exp(-c_1(d)\sqrt{m})$, the following holds uniformly for all $f,g \in \cF_{L}(\Omega')$:
    \begin{equation}\label{Eq:Bro}
        2^{-1}\int_{\Omega'} (f-g)^2 d\QQ - CL^2m^{-\frac{4}{d}} \leq \int (f-g)^2 d\QQ_m \leq 2\int_{\Omega'} (f-g)^2 d\QQ + CL^2m^{-\frac{4}{d}}.
    \end{equation}
    \end{lemma}
    
    Next, we introduce an estimator for the task of improper linear regression with ``sub-Gaussian tail guarantees'',  based on a recent result that is presented in \citep[Prop. 1]{mourtada2021distribution}. Its statistical aspects are proven in the seminal works \citep{tsybakov2003optimal,lugosi2019sub}, and guarantees on its runtime are given in \citep{hopkins2018sub,depersin2019robust,hopkins2020robust}.  \textit{Remarkably}, it would be crucial to our analysis.
    %  for more details see \citep{mourtada2021distribution}.
    
    \begin{lemma}{\label{Lem:LinearReg}}[\citep[Prop. 1]{mourtada2021distribution}]
     Let $m \geq d+1$, $d \geq 1$, $\delta \in (0,1)$ and $Z_1,\ldots,Z_m \underset{i.i.d.}{\sim} \QQ$, where $\QQ$ is supported on $\Omega' \subset B_d$ with \emph{a known} covariance matrix $\Sigma$. Consider the regression model 
     \[
        W = f^{*}(Z) + \xi,
    \]
    % $f^{*}$ is $L$-Lipschitz
     where  $\|f^{*}\|_{\infty} \leq L$.  Then, there exists an estimator that outputs an affine function with runtime of $\tilde{O}(m)$  \[\widehat{f}_{R,\delta}:(\Sigma,\{(Z_i,W_i)\}_{i=1}^{m}) \to \R^{d+1}\] 
     w that satisfies with probability at least $1-\delta$
     \[
         \int (\widehat{f}_{R,\delta}(x) - w^{*}(x))^2d\QQ(x) \leq \frac{C(\sigma+ L)^2(d+\log(1/\delta))}{m},
        %  +  \argmin_{ w \in \R^{d+1}}\int (w^{\top}(x,1)-f^{*}(x))^2d\QQ(x).
    \]
    where $w^{*} = \argmin_{w \text{ affine }}\int (w- f^{*})^2d\QQ$.
    \end{lemma}
    
    %  \[
    %      \int (\widehat{f}_{R,\delta}(x) - f^{*}(x))^2d\QQ(x) \leq \frac{C(\sigma+ L)^2(d+\log(1/\delta))}{m}+ \int (w_{*}^{\top}(x,1)-f^{*}(x))^2d\QQ(x)
    % \]
    
    % We remark that the bound in the last equation is tight, i.e. it is well known that any estimator for linear regression $\bar{w}$, we have that
    % \begin{equation}\label{Eq:tightEq}
    %          \frac{c_1d(\sigma+ L)^2}{m}+ \inf_{ w \in \R^{d+1}}\int (w^{\top}(x,1)-f^{*}(x))^2d\QQ(x) \leq \E \int (\bar{w} - f^{*})^2d\QQ.
    % \end{equation}
    
    % Also, we will need the following result on the (Lipschtitz convex LSE
    % \begin{lemma}
    % \end{lemma}
    % The proof this corollary is just to divide the data in to $C_1(l)\log(m)$ 
    \subsection{Proof of Theorem \ref{Thm:ConvexFunc}}
    The first ingredient in our estimator is our new approximation theorem for convex functions:
    
    %, which holds under our standing Assumption \ref{Ass:One} on $\PP$:
    % , and $X_1',\ldots,X_{4^{-1}k^\frac{d+2}{2}}' \underset{i.i.d.}{\sim} \PP$, 
    \begin{theorem}\label{Cor:Aproximation}
     Let $\Omega \subset B_d$ be a convex set, $f \in \cF_{L}(\Omega)$, and $k \geq (Cd)^{d/2}$, for some large enough $C \geq 0$. Then, there exists a convex set $\Omega_k \subset \Omega$ and a $k$-simplicial convex function $f_k: \Omega_k \to \R$ such that 
    \begin{equation}\label{Eq:support}
    \PP(\Omega \setminus \Omega_{k}) \lesssim d k^{-\frac{2(d+2)}{d(d+1)}}
    \end{equation}
    and 
    \begin{equation}\label{Eq:Error}
        \int_{\Omega_k}(f_k-f)^2d\PP \lesssimd L^2 \cdot k^{-\frac{4}{d}}.
    \end{equation}
    % and 
    % \[
    %     \|f_k-f^{*}\|_{\infty} \leq O_d(Lk^{-\frac{2(d+2)}{(d+1)}}).
    % \]
    %Furthermore,  we also have that
    %    (1-c_d\log(k)^{d-1}k^{-\frac{d+2}{d}})
    %  and the Hausdorff distance $d_{H}(\Omega  ,\Omega_k)$ is bounded by $ O_d(k^{-\frac{d+2}{d}\frac{1}{d}})$.
    \end{theorem}
    
    % \gil{make sure that $\Omega_k$ part is correct}
    % \begin{remark}
    % Note that the approximation error of the above theorem does not depend on $\Omega$. 
    Note that both $\Omega_k$, as well as $f_k$, depend on $f^*$. The bound of  \eqref{Eq:Error} is in fact tight, up to a constant that only depends on $d$, cf. \citep{ludwig2006approximation}. 
    % We remark that the above theorem has a general  
    % We believe that a more careful analysis may remove the redundant $\log(k)^{2/(d+1)}$ factor. 
    % also when $\Omega$ is a polytope the $\log(k)^{2/(d+1)}$ can be removed. %Moreover, we cannot improve the bounds of  \eqref{Eq:support}. 
    Also note that $f_k$ is not necessarily an $L$-Lipschitz function, i.e., it may be an ``improper'' approximation to $f$.   
    For simplicity, we shall assume that $L = \sigma = 1$. Also, since our function is $1$-Lipschitz and $\Omega \subset B_d$, we may also assume that $\|f^{*}\|_{\infty} \leq 1$. Finally, we may and do assume that $\PP = U(\Omega)$, since we can always simulate $\Theta_d(n)$ uniform samples using the method of rejection sampling given samples from any distribution $\PP$ which satisfies Assumption \ref{Ass:One} (cf. \citep{devroye1986nonuniform}).
    % , an therefore will obtain the same rate as $n$ uniform samples
     
    Fix $n \geq (Cd)^{d/2}$, $f^{*} \in \cF_{1}(\Omega)$, and set $k(n) := n^{\frac{d}{d+4}}$. Let $f_{k(n)}: \Omega_{k(n)} \to \mathbb R$ be the convex function whose existence is guaranteed by Theorem \ref{Cor:Aproximation} for $f = f^{*}$. We have
    \begin{equation}\label{Eq:boundOne}
      \int (f^{*}-f_{k(n)})^2 d\PP \lesssimd n^{-\frac{4}{d+4}},
    \end{equation}
    and there exist $\triangle_1,\ldots,\triangle_{k(n)} \subset \Omega$ simplices such that $\left. f_{k(n)}\right|_{\triangle_i}$ is affine on each $i$.
    
    % First, let us assume $f^{*} = f_{k(n)}$, later we will remove this assumption.
    
    If we were given the decomposition of $\Omega$ into pieces on which $f^*$ is near-affine, it would be relatively simple to estimate $f^*$, as we show  in Appendix \ref{App:SimpleEstimator} below. We recommend reading it, to get some intuition for our approach, before attempting the description and correctness proof for our ``full'' estimator below.

    To overcome the fact that we do not know the simplices $\triangle_i$ on which $f_k$ is affine, we need another lemma, which says that if we randomly sample a set of $n$ points $\{X_{n+1}, \ldots, X_{2n}\}$ from $\Omega$, there exists a collection of at most $\Tilde O_{d}(k(n))$ simplices covering ``most'' of $\Omega_k$, such that the vertices of each simplex belong to $\{X_{n+1}, \ldots, X_{2n}\}$ and $f_k$ is affine on each simplex in the collection:
    
    % \begin{lemma}\label{Lem:RandomSimp}
    % 	Let $n \geq C_d$, and $\triangle_1,\ldots,\triangle_{k(n)}$ that are defined above, and  consider $X_{2n+1},\ldots,X_{3n}  \underset{i.i.d.}{\sim} \PP$. Then, with probability of at-least $1-n^{-1}$, there exists a set of disjoint simplices, denoted by $\cS_{X}$ such that:
    % 	\begin{itemize}
    % 		\item The vertices of each simplex in $\cS_{X}$ lie in $\{X_{2n+1},\ldots,X_{3n}\}$.
    % 		\item Any $\tilde{\triangle} \in \cS_{X}$ lie in one of the simplices  $\triangle_1,\ldots,\triangle_{k(n)}$.
    % 		\item For each $ 1 \leq i \leq k(n)$ there are at most $O_d(\ln(n)^{d-1})$ simplices in $\cS_{X}$ that lie in $\triangle_i$, such that measure of their union satisfies $ (1-O_d(\ln(n)^{d-1}n^{-1}))\PP(\triangle_i)$.
    % 	\end{itemize}
    % \end{lemma}
    % Remark \ref{rem:PR} that appears in the previous section, claims that our estimator works only in the random design setting,  is based on the above lemma.
    \begin{lemma}\label{Lem:RandomSimp}
    	Let $n \geq d+1$, and $\triangle_1,\ldots,\triangle_{k(n)}$ that are defined above, and let $X_{n+1},\ldots,X_{2n}  \underset{i.i.d.}{\sim} \PP$. Then, with probability at least $1 - n^{-1}$, there exist $k(n)$ disjoint sets $\cS_{X}^{1},\ldots \cS_{X}^{k(n)}$ of simplices with disjoint interiors such that  
    	\begin{enumerate}
     		\item The vertices of each simplex in $ \bigcup_{i=1}^{k(n)}\cS_{X}^{i}$ lie in $\{X_{n+1},\ldots,X_{2n}\}$. Moreover, for each $ 1\leq i \leq k(n)$, we have that  $|\cS_{X}^{i}| \lesssimd \log(n\PP(\triangle_i))^{d-1}$.
    % 		 Moreover, each $\tilde{\triangle} \in \cS_{X}$ lie in one of the simplices  $\triangle_1,\ldots,\triangle_{k(n)}$.
    % 		\item 
    		\item For each $ 1 \leq i \leq k(n)$, we have that $\bigcup \cS_{X}^{i} \subset \triangle_i$, and 
    % 		$\PP(\triangle_i) \geq C_d\log(n)/n$, where $C_d$ is large enough constant,
    		\[
    		\PP(\bigcup \cS_{X}^{i}) \geq \PP(\triangle_i)-\min\left\{O_d\left(\frac{\log(n)\log(n\PP(\triangle_i))^{d-1}}{n}\right),\PP(\triangle_i)\right\}.
    		\]
    	\end{enumerate}
    \end{lemma}
    The proof of this lemma appears in \S \ref{App:RandomSimp}.  Essentially, this lemma states that we can triangulate ``most'' of each simplex $\triangle_i$ with ``few'' simplices whose vertices lie among the points $X_{n+1},\ldots,X_{2n}$ which fall in $\triangle_i$, so long as $\triangle_i$ is large enough. From now on, we condition on the high-probability event of Lemma \ref{Lem:RandomSimp}.
    
    Note that if we were given the set of simplices $\mathcal S_{X}:= \bigcup_{i=1}^{k(n)}\cS_{X}^{i}$, we could use the same strategy as in Appendix \ref{App:SimpleEstimator} to obtain a minimax optimal estimator for this task as well. Unfortunately, we do not know how to identify the simplices of $\cS_{X}$, but we do know that they belong to the collection of all simplices with vertices in $\{X_{n + 1}, \ldots, X_{2n}\}$,
     \begin{equation}
         \cS:=\left\{\conv\{X_{n+i}:i \in S\}: S
    \subset [n], |S| = d+1 \right\}.
     \end{equation}
     Note that $|\cS| = O_d(n^{d+1})$, which is polynomial in $n$.
     
    Instead of trying to identify the simplices $\cS_X \subset \cS$ on which $f$ is close to being linear, our algorithm  finds a function $\widehat f$ which, on \emph{every} simplex $\triangle \in \cS$, is ``not much farther'' from the best linear approximation to $f$ on $\triangle$ then $f$ is. Since $f$ itself is close to its best linear approximation on each simplex in $\cS_X$, $\widehat f$ will be close to $f$ on $\bigcup \cS_X$, which is most of $\Omega$.
    
    We restate this a bit more precisely: if $\widehat{f}:\Omega \to \R$ is a convex Lipschitz function such that
    \begin{equation}\label{Eq:ideal_fhat}
        \forall \triangle \in \cS: \   \int (\widehat{f} - w^{*}_{\triangle})^2d\PP_{\triangle} \leq \Tilde{O}_d\left(\int (f^{*} - w^{*}_{\triangle})^2d\PP_{\triangle} + (\PP({\triangle})n)^{-1}\right)
    \end{equation}
    where $w^{*}_{\triangle} = \inf_{w \text{ affine}}\int(f^{*}-w)^2d\PP_{\triangle}$, then $\widehat{f}$ satisfies
    $\int_{\Omega} (\widehat{f}-f^{*})^2d\PP \leq \tilde{O}_d(n^{-\frac{4}{d+4}})$, and the RHS is the minimax optimal rate. (The idea of the proof is to use the fact that $f^*$ is close to $w_\triangle^*$ for each $\triangle \in \cS_X$, along with the triangle inequality, and then sum over all simplices in $\cS_X$; the full justification is given at \eqref{Eq:ftild_approx}-\eqref{Eq:lastEq} below.) In the remainder of this section, we will describe an efficient algorithm which constructs a function $\widehat f$ which comes ``close enough'' to satisfying \eqref{Eq:ideal_fhat} that it manages to attain the minimax optimal rate.
    
    %  theclaims that our estimator works only in the random design setting, is based on the above lemma.
    %  We remark \ref{Rem:PR} in the pr
    % Using Remark \ref{Rem:PowerOfApprox}, we may further assume without loss of generality that under the event of the lemaa the support $f_{k(n)}$ that is defined above, satisfies 
    % Using Lemma \ref{Lem:LoneChaining}, we can further assume
    % Clearly, now $f_{k(n)}$ depends on these data points but does not depend on $\cD$.evious section that
    %  is based on the main result of \citep{dwyer1988convex}, and it 

    % Another lemma that we  ingridiend  follows of Lemma \ref{Lem:LinearReg} for completeness its proof appears in Appendix \ref{Al:One}.
    
 We now begin the description of our algorithm. In the notation of \eqref{Eq:ideal_fhat}, for each simplex $\triangle \in \cS$, we estimate $w^{*}_{\triangle}$ by applying Lemma \ref{Lem:LinearReg}, with the data points of $\cD_1 := \{(X_i,Y_i)\}_{i=1}^{n/2}$ that lie in $\triangle$ as input; denote the regressor we obtain by $\widehat{w}_{\triangle}$.
 
 Next, we shall need to estimate (with high probability) the mean squared error of the regressor on each simplex in $\cS$, i.e.  $\ell_\triangle^2 := \|f^* - \widehat{w}_\triangle\|^2_{L^2(\triangle)}$, up to a polylogarithmic multiplicative factor, using the data points in $\cD_2 = \{(X_i,Y_i)\}_{i=n/2+1}^{n}$ that lie in $\triangle$. Letting $w^*_\triangle$ to be defined as above, we have 
    $$\ell_\triangle^2 = \|w^*_\triangle - \widehat w_\triangle\|^2_{L^2(\triangle)} + \|f^* - w_\triangle^*\|^2_{L^2(\triangle)};$$
    the first term is called the (squared) estimation error and the second is called the (squared) approximation error. 
    By Lemma \ref{Lem:LinearReg}, with probability at least $1-n^{-2d}$ the estimation error will be at most $Cd\log(n)/(\PP(\triangle)n)$, which is no more than a $O(d\log(n))$ factor times the expected estimation error. However, $f^{*}$ may not be affine on $\triangle$, and the squared approximation error may be significantly larger than the squared estimation error. When this occurs, the estimation of $\ell_\triangle^2$ by  noisy samples is challenging, even in the (unrealistic) setting of sub-Gaussian noise with known variance $\sigma^2$. Indeed, it would be natural to estimate the approximation error by the (centered) empirical mean of the squared loss, namely
    \[
        \frac{1}{\PP(\triangle)n}\sum_{(X,Y) \in \cD_2 , X \in \triangle}(Y - \widehat{w}_{\triangle}(X))^2 - \sigma^2.
    \]
     However, the additive deviation of this estimate is of order $\Omega_d(n_\triangle^{-\frac{1}{2}})$, where $n_\triangle \approx \PP(\triangle)n$ is the number of data points falling in $\triangle$, and therefore when $\ell_\triangle^2$ is in the range 
    % that is essential for our algorithm to succeed.
    $[ O_d(n_\triangle^{-1}),\Omega_d(n_\triangle^{-\frac{1}{2}})]$
    we will not be able to estimate $\ell_{\triangle}^2$ even \emph{up to a multiplicative constant}, which is what our algorithm requires in order to succeed.
    % on $\triangle$, i.e.  we  need to estimate , where $w^*$ is the best affine approximation to $f^{*}$. 
    % By orthogonality, we know that
    % $
    %     \int (w^{*}-f^{*})^2 d\PP_{\triangle} \leq \int (\widehat{w}_{\triangle}-f^{*})^2 d\PP_{\triangle}, 
    % $
    % and moreover, by Lemma \ref{Lem:LinearReg} and the orthogonality  we also have 
    % \[
    %      \int (\widehat{w}_{\triangle}-f^{*})^2 d\PP_{\triangle} 
    %     \leq \int (w^{*}-f^{*})^2 d\PP_{\triangle} + Cd\log(n)/(\PP(\triangle)n). Therefore, up to an additive deviation of $\tilde{O}_d(1/(\PP(\triangle)n)))$, one may estimate $\int (\widehat{w}_{\triangle}-f^{*})^2 d\PP_{\triangle}$. 
    % \]
    % that is the $L_2^2(\PP_{\triangle})$-norm of the \emph{convex} function $f^*-\widehat{w}_{\triangle}$,
    
    To overcome this problem necessitates constructing a new efficient procedure that uses the data points of $\cD_2$ that fall in $\triangle$ to estimate $\ell_{\triangle}^2$ up to a multiplicative constant with an \emph{additive} deviation of $\tilde O_d((\PP(\triangle) n)^{-1})$. Using the convexity of $f^*-\widehat{w}_{\triangle}$ and techniques from potential theory and stochastic geometry, we show that such a procedure is indeed possible.
      % and furthermore this estimate has to to valid  with probability of at least $1- n^{-2d}$.
    %The formulation of the next lemma is adjusted to the proof of Theorem \ref{Thm:ConvexFunc}.
    
    \begin{lemma}\label{Lem:VerMSE}
    Let $\triangle \subset \Omega$ and let $g: \triangle \to [-2L,2L]$ be a convex function. Then, there exists an estimator $\widehat{f}_{E,\delta(n)}$ (that uses that points of $\cD_2$) with a runtime $O_d(n^{O(1)})$ that satisfies     with probability at least $1 - n^{-2d}$
    \begin{align*} 
     \|g\|_{L_2(\triangle)}^2 &\le  \widehat f_{E, \delta} \lesssimd \log(L)\log(n)^{2d-1}\|g\|^2_{L_2(\triangle)}
    \end{align*}
    when $\|g\|_{L_2(\triangle)}^2 \gtrsim (L + \sigma)^2 \frac{\log(2/\delta) }{\PP(\triangle)n}$.
    %  $w_b = \argmin_{w \text{ affine }} \|g-w_{b}\|_{2}$, and
    \end{lemma}
    \noindent{We remark that our estimator requires an upper bound on $L$ and $\sigma$ (up to multiplicative constants that only depend on $d$). Both can be found using standard methods when $L= \Theta(\sigma) =\Theta(1)$.}  
    We denote the output of the estimator of Lemma \ref{Lem:VerMSE} for $g = f^*-\widehat{w}_{\triangle}$ by $\widehat{\ell}^2_{\triangle}$. Given our regressors $\widehat w_\triangle$ and squared error estimates $\widehat \ell^2_\triangle$, we proceed to solve the quadratic program which encodes the conditions $\|\tilde f - \widehat{w}_{\triangle}\|^2_{L^2(\triangle)} \le \widehat \ell^2_\triangle$ for all simplices with large enough volume. (We rely on the fact that the $L^2$-norm on each simplex can be approximated by the empirical $L^2$-norm, again using Lemma \ref{Lem:Bro}.) This program is feasible, since $f^*$ itself is a solution. $\tilde f$ is close to $f^*$ on every simplex in our collection and in particular on the simplices restricted to which $f^*$ is near-affine (which we don't know how to identify), which allows us to conclude that 
    \[
    \int_{\tilde\Omega_{k(n)}}(\tilde{f} - f^{*})^2d\PP \leq \tilde{O}_d(n^{-\frac{4}{d+4}})
    \]
    with high probability, where $\tilde \Omega_{k(n)}$ is the union of the simplices in Lemma \ref{Lem:RandomSimp}.
    
    %  To solve this problem, we use the fact that $\Omega_{k(n)}$ contains another random set we construct, $\Omega_X$, with high probability. We simply run the convex LSE on $\Omega \setminus \Omega_{X}$, where $\tilde{k}(n) = c_d(k(n)/\log(k(n)))^{\frac{(d+2)}{d}}$; this is sufficient because $\Omega \backslash \Omega_X$ has small volume (see Lemma \ref{Lem:UseTheFloaingBody} below).
    
    So we have constructed a function $\tilde f$ which closely approximates $f^*$ on $\tilde\Omega_{k(n)}$. $\tilde \Omega_{k(n)}$ is not known to us, but as we shall see, $\Omega \backslash \tilde \Omega_{k(n)}$ has asymptotically negligible volume, so the function $\min\{\tilde{f}, 1\}$ turns out to be a minimax optimal improper estimator (up to logarithmic factors) of $f^*$ on all of $\Omega$. In order to transform this improper estimator to a proper estimator, i.e., one whose output is a convex $1$-Lipschitz function, we use a standard procedure (denoted by $MP$), as described in Appendix \ref{App:MP} below. This concludes the sketch of our algorithm.
    
    % In addition, on the points the lie in $\Omega_{X}^{c}$, we will apply the convex LSE
    
    %On the set $\Omega_{X}$, we find a convex function has the same estimated MSEs over all the simplices in $\cS$ as their corresponding regressors that we found earlier. 
    
     %that will let us have a ``proper" MSE up to a $Cd\log(n)$ factor. 
    % Now, we know that the underlying  $f^{}$

    %  For each $X_{n+i} \in A$ choose the simplex from  $\cS$ that contains it that has the lowest MSE, and use its regressor to assign a value to $\tilde{f}(X_{n+i})$. For the remaining points, i.e. $X_{n+i} \notin A$  will be handled in a somehow similar fashion as the points of $\Omega \setminus \Omega_{k(n)}$ in Step I (see  \eqref{Eq:Boundary}).  We will show that that such an $\tilde{f}$ satisfies the equation above.
    % ********************
    % function is close to be affine on each one them based 
    % using the powerful estimator of Lemma \ref{Lem:LinearReg}. 
    % Since, by the previous lemma we know that there is a set of ``good" simplices that contains most o
    Pseudocode for the algorithm is given in Algorithm~\ref{alg:one} below. In its formulation, note that the procedure $\widehat f_{R, \delta(n)}$ is described in Lemma~\ref{Lem:LinearReg}, $\widehat f_{E, \delta(n)}$ is described in Lemma~\ref{Lem:VerMSE}, and $MP$ is described in Appendix~\ref{App:MP}. 
    
    \begin{algorithm}[H]
    \caption{ Minimax Optimal For $L$-Lipschitz Multivariate Convex Regression}\label{alg:one}
    \begin{algorithmic}
     \Require $\cD = \{(X_i,Y_i)\}_{i=1}^{n} = \cD_1 \cup \cD_2$
    \Ensure  A random $ \widehat{f}_{L} \in \cF_{L}(\Omega)$ s.t. w.h.p. $\|\widehat{f}_{L} - f^{*} \|_{\PP}^2 \leq  \tilde{O}_d((L+\sigma)^2n^{-\frac{4}{d+4}})$.  
    \State Draw $X_{n+1},\ldots,X_{2n} \underset{i.i.d.}{\sim} \PP$
    \State $\cS \gets \left\{\conv\{X_{n+i}:i \in S\}: S
    \subset [n], |S| = d+1 \right\}$
    \State \textcolor{blue}{Part I:} 
   
    % \State $\cS_2 \gets \emptyset$
    % \State Sort the simplices of $\cS$ in descending way w.r.t. by their $d$-volume.
    % \State For the simplices that their volume is less than $Cd\log(n)$.
    % $\triangle_1,\ldots, \triangle_{|\cS|}$ (sorted )
    \State $\delta(n) \gets n^{-(d + 2)}$
 
    % \State $\cS_2 \gets \emptyset$
    % \State Sort the simplices of $\cS$ in descending way w.r.t. by their $d$-volume.
    % \State For the simplices that their volume is less than $Cd\log(n)$.
    % $\triangle_1,\ldots, \triangle_{|\cS|}$ (sorted )
    % \State
   \For{$\triangle_1,\ldots,\triangle_i,\ldots \in \cS$} 
    
    % \Comment{Keep the simplices that $f^{*}$ is affine}
    
        % \State Denote by $\cD_i^{1} := \{ (X,Y) \in \cD^1: X \in \triangle_i \}$.
        % \State Let $l_i^2$ to be the output of the procedure $\widehat{f}_{E}$ using the with the date-points $\cD_i$.
         \State $\widehat{w}_i  \gets \widehat{f}_{R,\delta(n)}(\{ (X,Y) \in \cD^1: X \in \triangle_i \})$.
         \State $\widehat\ell_i \gets \min(4, \widehat{f}_{E,\delta(n)}(\{(X,Y-\widehat{w}_i(X)): (X,Y) \in \cD^2,X \in \triangle_i\}))$
        %  \State Let $\widehat{\ell}_i^2 = \widehat{f}_{E,\delta(n)}^2 + Cd\log(n)(L + \sigma)^2/(\PP(\triangle_i)n) $.
        %  \vspace{2mm}
        %  \State Set $\widehat{l}_i^2 \gets \widehat a_i^2  + (L+\sigma)^2 d\log(n)/(\PP(\triangle_i)n)$.
        %  \vspace{2mm}
    %    \State Replace entry of $\triangle_i$ by $(\triangle_i,l_i^2,\widehat{w}_i)$.
        % \State
   \EndFor
    % \State For the remaining simplices that their volume is less than $Cd\log(n)/n$ update  their entry in $\cS$ to $(S,1,\vec{0})$.
     
    % \State Draw $Z_1 \sim U(S_{1}),\ldots ,Z_i \sim U(S_{i}),\ldots,Z_{\binom{n}{d+1}} \sim U(S_{\binom{n}{d+1}})$
\State \textcolor{blue}{Part II:}
    \For{$ i \in 1,\ldots |\cS|$}
    
    % \If{$X_i \in A$}
    % {
    % \Comment{Assign values to the data points $X_1,\ldots,X_n$}
    \State Draw $Z_{i,1},\ldots,Z_{i,n} \sim \PP_{\triangle_i}$
    % \State $\cS_i \gets \{(\triangle,\widehat{w},l^2) \in \cS:  X_{n+i} \in \triangle  \}$ 
    % \State $\cS_i^{2} \gets \{(\triangle,\widehat{w},l^2+\|X_{n+i}-Z_{I(\triangle)}\|^2): (\triangle,\widehat{w},l^2) \in \cS,X_{n+i} \notin \triangle\}$
    % \State Set $\tilde{w}_i$ to be the regressor of the  lowest  third entry in $\cS_i$.
    \State Define an inequality constraint $I_j:= \frac{1}{n}\sum_{j=1}^{n}(f(Z_{i,j})-\widehat{w}_i^{\top}(Z_{i,j},1))^2 \leq \widehat{\ell}_i^2 + CL^2\sqrt{\frac{d\log(n)}{n}}$.
    % }
    % \Else
    % {
    % $\tilde{f}(X_{n+i}):= \widehat{f}_{A^c}(X_{n+i})$.
    % }
    \EndFor
    % \For{$ i \in 1,\ldots n$ }
    % {
    % \State Update the regressor $\widehat{w}_{n+i}$ to  $\widehat{w}_{\argmin_{1 \leq j \leq n}\tilde{l}_j^2 + \|X_{n+j}-X_{n+i}\|}$.
    % \State Set $\tilde{f}(X_{n+i}) := \widehat{w}_{n+i}^{\top}(X_{n+i},1)$
    % }
    % \State
    \State Construct $\tilde{f} \in \cF_{L}(\Omega)$ satisfying the constraints $I_1,I_2\ldots,I_{|S|}$ (cf. Eqs. \eqref{Eq:ConvCons1}-\eqref{Eq:ConvCons3})
    % \State Let $\widehat{f}_{c}$ be the convex LSE with the input $\{ (X,Y) : (X,Y) \in \cD, X \notin \Omega_{X} \}$.
    %  \State \textcolor{blue}{When $\Omega$ is a polytope: } \Return  $MP(\min\{\tilde{f},L\})$.
    %  \State \textcolor{blue}{Handling the boundary:} 
    %  \State \textcolor{red}{Under our assumptions on $\Omega$ we just Return }
    \State 
     \Return  $MP(\min\{\tilde{f},L\})$,
        % \State
       
        % \State Obtain the regressor $(\bar{w}_i,b_i)$ and  $l_i^2 \gets \frac{1}{|\cD_{i}|}\sum_{j=1}^{|\cD_{i}|}(Y_j - \bar{w_i}^{\top}(X_i,1))^2 -\sigma^2$ 
        
        % \If{$l_i^2 \geq \frac{2dL}{n\vol(\triangle_i)}$  or $\vol(\triangle_i) \leq (3dL)/n$}
        % {
        % \State $\cS \gets \cS \setminus 
        % \{(\triangle_i,w_i,l_i^2)\}$. 
        % }
    % \State Remove all the entries from $\cS$ that have a $l \geq c\log(n)^{-1}$. 
    % \State Draw $X_{n+1},\ldots,X_{2n} \underset{i.i.d.}{\sim} \PP$.
    % \State
    % \For{$ i \in 1,\ldots n$ }
    % {
    % % \Comment{Assign values to the data points $X_1,\ldots,X_n$}
    % \State $\cS_i \gets \{(\triangle,\widehat{w},l^2) \in \cS:  X_{n+i} \in int(\triangle)  \}$
    % \State Set $(\widehat{w}_{n+i},\tilde{l_{i}}^2)$ to be the regressor of the  lowest $l^2$ in $\cS_i$.
    % }
    % \State
    % \For{$ i \in 1,\ldots n$ }
    % {
    % \State Update the regressor $\widehat{w}_{n+i}$ to  $\widehat{w}_{\argmin_{1 \leq j \leq n}\tilde{l}_j^2 + L\|X_{n+j}-X_{n+i}\|}$.
    % \State Set $\tilde{f}(X_{n+i}) := \widehat{w}_{n+i}^{\top}(X_{n+i},1)$
    % }
    % \State
    % \Return Convex LSE with input $\{(X_{n+1},Y_{n+1}),\ldots,(X_{2n},Y_{2n})\}$. 
    \end{algorithmic}
    \end{algorithm}
    We now turn to the proof that Algorithm~\ref{alg:one} succeeds with high probability. In the analysis, we assume for simplicity that $L =\sigma = 1$. Let $\mathcal S$ be as defined in  Algorithm~\ref{alg:one}, and let 
    \[
    \cS^{T}:=\{\triangle: \triangle \in  \cS, \int_{\triangle} g^2d\PP \geq Cd\log(n)^2/n\},
    \]
    for some sufficiently large $C \geq 0$. In particular, we have $\PP(S) \geq C_1(C)d\log(n)/n$ for all $S \in \cS^{T}$. We first note that our samples may be assumed to be close to uniformly distributed on the simplices in $\cS^T$. Indeed, by standard concentration bounds, we have
    \begin{equation}\label{Eq:Sets}
        \forall \triangle \in \cS^T, j \in \{1, 2, 3\}:\quad \frac{1}{2} \leq \frac{\PP_n^{(j)}(\triangle)}{\PP(\triangle)} \leq 2,
    \end{equation}
    with probability $1 - 3 n^{-3d}$,  where $\PP_n^{(1)} = \frac{2}{n} \sum_{i = 1}^{n/2} \delta_{X_i}$, $\PP_n^{(2)} =  \frac{2}{n} \sum_{i = n/2 + 1}^{n} \delta_{X_i}$ and $\PP_n^{(3)} =  \frac{1}{n} \sum_{i = n + 1}^{2n} \delta_{X_i}$ (see Lemma \ref{Lem:Standard} in \S \ref{App:RandomSimp}). From now on, we condition on the intersection of the events of \eqref{Eq:Sets} and Lemma \ref{Lem:RandomSimp}. 
    The first step in the algorithm is to apply the estimator of Lemma \ref{Lem:LinearReg} for each $\triangle_i \in \cS^T$ with $\QQ := \PP_{\triangle_i}$, and $\delta = n^{-(d+2)}$, using those points among of $\cD_1$ that fall in $\triangle_i$. (By the preceding paragraph, under our conditioning, we may assume $\PP(\triangle_i) n$ of the points in $X_1, \ldots, X_{\frac{n}{2}}$ fall in each $\triangle_i$, up to absolute constants. We will silently use the same argument several more times below.) By the lemma and a union bound, we know that the following event has probability at least $1-n^{-1}$:
    % \widehat{w}_1,\ldots,\widehat{w}_i,\ldots,\widehat{w}_{|\cS^{T}|
    \begin{equation}\label{Eq:Regi}
    \begin{aligned}
    % \leq l_j^2 
    %  ($ 1 \leq j \leq \binom{n}{d+1}$)
           \forall 1 \leq i \leq |\cS^{T}| \ \ \int_{\triangle_i} (\widehat{w}_i - f^*)^2 d\PP_{\triangle_i} &\leq  \frac{2Cd\log(n)}{\PP(\triangle_i)n} + \int_{\triangle_i} (w_{i}^* - f^*)^2 d\PP_{\triangle_i} ,
        %   \int_{\triangle_i}( - f^{*})^2d\frac{\PP}{\PP(\triangle_i)} + \frac{Cd\log(n)}{\PP_n(\triangle_i)n} \\&\leq
        % \int_{\triangle_i}(w_{i,*}^{\top}(x,1) - f^{*})^2d\frac{\PP}{\PP(\triangle_i)} +
    \end{aligned}
    \end{equation}
    where $w_{i}^* = \argmin_{\text{$w$ affine}}\int_{\triangle_i}(w(x) - f^{*})^2d\PP_{\triangle_i}$. We condition also on the event of \eqref{Eq:Regi}. Next, we apply Lemma \ref{Lem:VerMSE} (with $\delta = n^{-(d+2)}$) on each $\triangle_i$, with $g = f^{*} - \widehat{w}_i$, and using those points among of $\cD_2$ that fall in $\triangle_i$, and obtain that
    \begin{equation}\label{Eq:Condition}
    \forall 1 \leq i \leq |\cS|: \ \ \int_{\triangle_i} (\widehat{w}_i - f^{*})^2 d\PP_{\triangle_i} \leq \widehat{\ell}_i^2,
    \end{equation}
    with $\widehat \ell_i^2$ as defined in Algorithm \ref{alg:one}. Note that for $\triangle \in \cS \setminus \cS^{T}$, taking $\widehat w_i = 0$ suffices, since $f^*$ is bounded by $1$, the loss is bounded by $4$. 
    % \begin{lemma}
    Finally, we further condition on the event of the last equation.

    We proceed to explain and analyze Part II of Algorithm \ref{alg:one}. We first claim that conditioned on \eqref{Eq:Condition}, the function $f^{*}$ satisfies the constraints $I_1, I_2, \ldots$ defined in the algorithm  with probability at least $1 - n^{-1}$. Indeed, for each $ 1 \leq i \leq |\cS|$, $\|(f^{*}- \widehat{w}_i)^2\|_{L^\infty(\triangle_i)} \leq 4$, so by Hoeffding's inequality and \eqref{Eq:Condition} we know that with probability at least $1-n^{-(d+2)}$, we have that
    \begin{equation}\label{Eq:hoeff}
    \begin{aligned}
    % ))^2 \int(f-\widehat{w}_i^{\top}(z,1))^2 d\QQ_{k(n)}
        \frac{1}{n}\sum_{j=1}^{n}(f^{*}(Z_{i,j})-\widehat{w}_i(Z_{i,j}))^2 &\leq \int_{\triangle_i} (f^{*}-\widehat{w}_i)^2 d\PP_{\triangle_i} +\sqrt{\frac{ Cd\log(n)}{n}} \\&\leq \widehat{\ell}_i^2 +  \sqrt{\frac{Cd\log(n)}{n}}.
    \end{aligned}
    \end{equation}
    Taking a union over $i$, we know that \eqref{Eq:hoeff} holds for all $i$ with probability at least $1 - n^{-1}$.
    
    We also note (for later use) that applying Lemma \ref{Lem:Bro} to the measures $\PP_{\triangle_i}$ and using a union bound, it holds with probability at least $1 - C n^d e^{-c\sqrt n}$ that for all $i$, the empirical measure $\PP_{\triangle_i, n} = \frac{1}{n} \sum_{j = 1}^n \delta_{Z_{i, j}}$ on $\triangle_i$ approximates $\PP_{\triangle_i}$ in the sense of \eqref{Eq:Bro}. We condition on the intersection of these two events as well. 
    
    % \begin{equation}
    % \begin{aligned}
    % & \quad \quad \quad \quad \quad \quad \quad \quad \quad \quad \min_{f(Z_{i,j}),\nabla f(Z_{i,j})} \quad 1 \\&
    % \quad \textrm{s.t.} \quad \forall \  1\leq i \leq |\cS^{T}| \ \  \frac{1}{k(n)}\sum_{j=1}^{k(n)}(f(Z_{i,j})-\widehat{w}_iZ_{i,j})^2 \leq \widehat{l}_i^2 + \frac{CL^2d\log(n)}{\sqrt{k(n)}}
    % \\& \quad \quad \quad  \forall (i_1,j_1),(i_2,j_2) \in ([|\cS^{T}|],[k(n)]) \quad  |f(Z_{i_1,j_1}) - f(Z_{i_2,j_2})| \leq L\| Z_{i_1,j_1} -Z_{i_2,j_2} \|
    % \\& \quad \quad \quad  \forall (i_1,j_1),(i_2,j_2) \in ([|\cS^{T}|],[k(n)]) \quad f(Z_{i_2,j_2}) \geq \nabla f(Z_{i_1,j_1})^{\top}(Z_{i_2,j_2}-Z_{i_1,j_1})
    % \end{aligned}
    % \end{equation}
    
    We now explain how to algorithmically construct $\tilde f \in \mathcal F_L(\Omega)$ satisfying all the constraints $I_j$. The idea is to mimic the computation of the convex LSE \citep{seijo2011nonparametric}, by considering the values of the unknown function $y_{i,j} = \tilde f(Z_{i,j})$ and the subgradients $\xi_{i,j} \in \partial f(Z_{i, j})$ at each $Z_{i,j}$ as variables. More precisely, we search for $y_{i,j} \in \mathbb R$ and $\xi_{i,j} \in \mathbb R^d$ satisfying the following set of constraints (here $L=1$):
    \begin{align}%\label{Eq:ConvCons}
    \forall i \le |\cS|: \qquad & \frac{1}{n}\sum_{j=1}^{n}(y_{i,j}-\widehat{w}_i(Z_{i,j}))^2 \leq \widehat{\ell}_i^2 + \sqrt{\frac{Cd\log(n)}{n}} \label{Eq:ConvCons1}
    \\ 
    \forall (i, j) \in [|\cS|] \times [n]: \qquad & \|\xi_{i,j}\|^2 \le L^2 \label{Eq:ConvCons2}
    %\\& \forall (i_1,j_1),(i_2,j_2) \in [|\cS|] \times [n] \quad  |y_{i_1,j_1}) - y_{i_2,j_2}| \leq L\| Z_{i_1,j_1} -Z_{i_2,j_2} \|
    \\   \forall (i_1,j_1),(i_2,j_2) \in [|\cS|] \times [n]: \qquad & y_{i_2,j_2} \geq \langle \xi_{i_1, j_1}, Z_{i_2,j_2}-Z_{i_1,j_1} + y_{i_1, j_1} \rangle.
    \label{Eq:ConvCons3}
    \end{align}
    
    For any feasible solution $(y_{i,j}, \xi_{i,j})_{i, j}$ of \eqref{Eq:ConvCons1}-\eqref{Eq:ConvCons3}, define the affine functions $a_{i, j}(x) = y_{i,j} + \langle \xi_{i, j}, x - Z_{i, j}\rangle$. We claim that the function $\tilde f = \max_{i, j} a_{i, j}$ is a $1$-Lipschitz convex function which satisfies the constraints $I_j$. Indeed, \eqref{Eq:ConvCons3} guarantees that $\tilde f(Z_{i,j}) = y_{i,j}$ for each $i$, so the $I_j$ are satisfied due to \eqref{Eq:ConvCons1}; moreover, the $a_{i,j}$ are convex and $1$-Lipschitz (the latter because of \eqref{Eq:ConvCons2}), so $\tilde f$ is convex as a maximum of convex functions and $1$-Lipschitz as a maximum of $1$-Lipschitz functions. %(To prove the latter assertion, for any two points $x, y \in \Omega$ take the line segment between $x$ and $y$ and break it up into segments $L_k = (x_k, x_{k + 1})$ such that $\tilde f(x)$ coincides with one of the $w_{i, j}$ on all of $L_k$; then $|f(x_{k + 1}) - f(x_k)| \le |x_{k + 1} - x_k|$. Summing over $k$ gives the desired inequality $|f(y) - f(x)| \le |y - x|$.) 
    
    Conditioned on \eqref{Eq:hoeff} there exists a feasible solution to the problem \eqref{Eq:ConvCons1}-\eqref{Eq:ConvCons3}, namely that obtained by taking $y_{i,j} = f^*(Z_{i, j})$, $\xi_{i,j} \in \partial f^*(Z_{i,j})$ (where $\partial f(x)$ denotes the subgradient set of a convex function $f$ at the point $x$). Moreover, the constraints in \eqref{Eq:ConvCons1}-\eqref{Eq:ConvCons3} are either linear or convex and quadratic in $f(Z_{i,j}), u_{i, j}$, and hence the problem can be solved efficiently. For instance, it can be expressed as a second-order cone program (SOCP) with $n^{2d + 2}$ variables and constraints, which can be solved in time $n^{O(d)}$ (see, e.g., \cite{ben2001lectures}).
    
    % Denote the solution of this SOCP problem by $\tilde{f}$.
    % Now, note that the sets in $\cS_{X}$ are disjoint (up a measure zero), we know that by the definition of $A$ that for each set $\tilde{\triangle} \in \cS_{X}$ that  $\PP( \tilde{\triangle} \cap A) \in \{0,\PP( \tilde{\triangle})\}$. Therefore, we conclude that for all ****.
    Next, recall that under our conditions on the $Z_{i, j}$, we have for each $i$ that
    \begin{equation}\label{Eq:Bro_ftild}
        \int_{\triangle_i} (\tilde{f}(x) - f^*(x))^2 d\PP_{\triangle_i} \leq \frac{1}{n}\sum_{j = 1}^n (\tilde{f}(Z_{i,j}) - f^*(Z_{i, j}))^2 + C n^{-\frac{4}{d}},
    \end{equation}
    since both $f^*$ and $\tilde f$ lie in $\mathcal F_1(\Omega)$. 
    Recall also that under our conditioning, for each $i$, the constraint
    $$\frac{1}{n}\sum_{j=1}^{n}(y_{i,j}-\widehat{w}_i(Z_{i,j}))^2 \leq \widehat{\ell}_i^2 + \sqrt{\frac{Cd\log(n)}{n}}$$
    holds whether we take $y_{i, j} = \tilde f(Z_{i, j})$ or $y_{i, j} = f^*(Z_{i, j})$. Using this bound along with the inequality $(f^* - \tilde f)^2 \le 2 (\tilde f - \widehat w_i)^2 + 2(\widehat w_i - f^*)^2$ in \eqref{Eq:Bro_ftild}, we obtain
    \begin{equation}\label{Eq:ftild_approx}
        \int_{\triangle_i} (\tilde{f}(x) - f^*(x))^2 d\PP_{\triangle_i} \leq 2 \widehat \ell_i^2 + \sqrt{\frac{Cd\log(n)}{n}} + C n^{-\frac{4}{d}} \le 2\widehat \ell_i^2 + C' n^{-\frac{4}{d+4}},
    \end{equation}
    where we used our assumption of $d \geq 5$. 
    Now, recalling that $\ell_i^2$ denotes the LSE error $\|f^* - \widehat{w}_i\|^2_{L^2(\triangle_i)}$, which is bounded by \eqref{Eq:Regi}, we have 
    \[
    \widehat \ell_i^2 \lesssimd \log(n)^{3d}\ell_i^2 \lesssimd \log(n)^{2d-1}\left(\|f^* - w_i^*\|_{L^2(\triangle_i)}^2+ \frac{C(d)\log n}{\PP(\triangle_i) n}\right).
    % C_d \log(n)^{2d-1}\left(a_i^2 + \frac{C(d)\log n}{\PP(\triangle_i) n}\right) =
    \]
    % and that by the upper bound in Lemma \ref{Lem:VerMSE}, 
    % $$\widehat a_i \le C_d' \|f^* - w_i^*\|_{L^2(\triangle_i)} + \left(\frac{Cd\log n}{\PP(\triangle_i) n}\right)^{\frac{1}{2}},$$
    % by same computations as we used above in deriving. 
    %Again, using the inequality $(a + b)^2 \le 2a^2 + 2b^2$ and 
    Substituting in \eqref{Eq:ftild_approx}, we obtain for any $\triangle_i$ that 
    \begin{equation}\label{Eq:ftild_li}
        \int_{\triangle_i} (\tilde{f}(x) - f^*(x))^2 d\PP_{\triangle_i} \lesssimd \log(n)^{2d-1} \|f^* - w_i^*\|_{L^2(\triangle_i)}^2 + \frac{\log (n)^{2d}}{\PP(\triangle_i) n} +n^{-\frac{4}{d+4}}.
    \end{equation}

    %along with \eqref{Eq:Condition} and the definition of $\widehat l_i^2$, we obtain that
    %\begin{equation}\label{Eq:ftild_approx}
    %    \E [\int_{\triangle_i} (\tilde{f} - f^{*})^2 d\PP  |\cE_1] \leq 4 \inf_{w \in \R^{d+1}}\int_{\triangle_i} (f^{*}- w^{\top}(x,1))^2 d\PP +  \frac{C_1d\log(n)}{n} +  C_3d \cdot \PP(\triangle_i)n^{-\frac{4}{d+4}}.
    %\end{equation}
    
    Now recall that $f_{k(n)}$ is our $k(n)$-simplicial approximation to $f$, and that $\left.f_{k(n)}\right|_{S_{i, j}}$ is affine for each $i$ and $S_{i, j} \in \cS^i_X$, where the sets $\cS_{X}^{i}$ are defined in Lemma \ref{Lem:RandomSimp}. Define $\tilde{\Omega}_{k(n)} := \bigcup\bigcup_{i=1}^{k(n)}\cS_{X}^{i}$. Recall that by definition,
    $\|f^* - w_m^*\|^2 = \inf_{\text{$w$ affine}} \|f^* - w\|^2_{L^2(S_m)}$ for any $S_m \in \cS$, in particular for those $S_m$ which belong to one of the $\cS^i_X$. Hence, multiplying \eqref{Eq:ftild_li} by $\PP(\triangle_i)$ and summing over all the $\triangle_i$ belonging to any of the $\cS^i_X$, we obtain
    \begin{equation}
        \begin{aligned}
            \int_{\tilde{\Omega}_{k(n)}} (\tilde{f} - f^{*})^2 d\PP &\lesssimd \sum_{\tilde\triangle \in \bigcup_{i=1}^{k(n)}\cS_{X}^{i}}\left( \log(n)^{2d-1}\inf_{\text{$w$ affine}}\int_{\tilde \triangle} (f^{*}- w)^2 d\PP + \frac{\log(n)^{2d}}{n} + \PP(\tilde\triangle)n^{-\frac{4}{d}}\right)
            \\&\lesssimd \log(n)^{2d-1}\int_{\tilde{\Omega}_{k(n)}} (f_{k(n)} - f^{*})^2d\PP + \log(n)^{3d}n^{-\frac{4}{d+4}} \\
            &\lesssimd n^{-\frac{4}{d+4}}\log(n)^{3d},
        \end{aligned}
    \end{equation}
    % \[  
    %      \E [\int_{\tilde{\triangle}} (\tilde{f} - f^{*})^2 d\PP  |\cE] \leq  4\inf_{w \in \R^{d+1}} \int_{\tilde{\triangle}}(f^{*}(x) - \widehat{w}_i^{\top}(x,1))^2d\PP + \frac{C_1d\log(n)}{n} + 2\PP(\tilde{\triangle})n^{-\frac{4}{d+4}}).
    % \]
    where we used the fact that the cardinality of $\bigcup_{i=1}^{k(n)}\cS_{X}^{i}$ is bounded order of $n^{\frac{d}{d+4}}\log(n)^{d}$, the disjointness of all the simplices comprising $\tilde\Omega_{k(n)}$, and finally the definition of $f_{k(n)}$. Next, it is not hard to show that $\|\tilde f\|_{\infty} \leq C$ (simply because $\tilde f$ is $1$-Lipschitz, $\Omega$ is contained in the unit ball, and $\int_{\Omega_{k(n)}}(\tilde{f} - f^{*})^2 \leq 4$). Thus, we obtain that  
    \begin{align*}
         \int_{\Omega_{k(n)} \setminus \tilde{\Omega}_{k(n)}} (\tilde{f} - f^{*})^2 d\PP &\leq \|\tilde{f}\|_{\infty}\PP(\Omega_{k(n)} \setminus \tilde{\Omega}_{k(n)}) \lesssim \sum_{i=1}^{k(n)}\PP(\triangle_i \setminus \bigcup \cS_{X}^{i})\lesssimd n^{-\frac{4}{d+4}}\log(n)^{d}, 
        %  \|O_d(L^2n^{-\frac{4}{d+4}}\log(n)^{d}),
    \end{align*}
    where we used part (2) of Lemma \ref{Lem:RandomSimp}.  Combining the last two equations, we obtain that 
    \begin{equation}\label{Eq:lastEq}
        \int_{\Omega_{k(n)}}(\tilde{f} - f^{*})^2 d\PP \lesssimd n^{-\frac{4}{d+4}}\log(n)^{3d}.
    \end{equation}
    Finally, since $\PP(\Omega \setminus \Omega_{k(n)}) \lesssimd n^{-\frac{4}{d+4}}$, we can estimate $f^*$ simply by $1$ on $\Omega \backslash \Omega_{k(n)}$  to obtain a minimax optimal estimator on all of $\Omega$ (recall that this is an improper estimator, and we can apply the procedure $MP$ described in Appendix \ref{App:MP} to obtain a proper estimator).  It is not hard to see that the runtime of the above algorithm is $n^{O(d)}$. The proof of Theorem \ref{Thm:ConvexFunc} is complete.
\subsection{On  Lemma \ref{Lem:VerMSE}: Optimal estimation of a norm of a convex function}
First, we develop a new estimator for the $L_1$ norm of any convex function.
    % In proof of the last lemma, we also prove the following result that may be indepdent of interest. 
    \begin{lemma}\label{Lem:LoneNorm}
  Let $\delta \in (0,1)$ and $f:K \to \R$ be a convex function, and suppose that $m$ i.i.d. samples are drawn from the regression model
    \[
        Y = f(Z) + \xi \text{ where $Z \sim U(K)$ and $\|f\|_{L_1(U(K))} \gtrsimd  (\sigma+\|f\|_{L_2(U(K))}) \cdot \sqrt{\frac{\log(2/\delta)}{m}}$.}
    \]
    There exists an estimator $\bar{f}_{m}^{\delta}$ which satisfies the following with probability at least $1-3\max\{\delta,e^{-cm}\}$:
    \[
        \bar{f}_m^{\delta} \asymp C(K) \cdot \|f\|_{L_1(U(K))},
    \]
    where $C(K)$ is a constant that only depends on the convex set $K$, and in particular dimension $d$.
    \end{lemma}
Clearly, Lemma \ref{Lem:LoneNorm} is invariant with respect to affine transformations of the domain. Thus, the constants $C(K)$ is the same for all $K$ in the class of affine images of a fixed convex body in $\R^d$, such as the class of simplices.
    
    The estimator of the last lemma gives an optimal error rate, with respect to the number of samples $m$, for the $L_1(U(\triangle))$-norm of any convex function $g$ (with no restriction on its uniform norm or Lipschitz constant). In the next step, we aim to find the $L_2(U(\triangle))$ norm of a convex function $g:\triangle \to [-L,L]$.

    To obtain a similar result as in the previous lemma, it will be essential for both the statistical guarantees and the computational aspects of the proposed estimator to assume that the domain of $g$ is a simplex and that $\|g\|_\infty \le L$. This part of the proof relies heavily on the notion of the floating body of a polytope \cite{schutt1990convex}, which is defined in the Appendix below.
\paragraph{On Lemma \ref{Lem:LoneNorm}}
Its main idea is the following: We introduce a probability measure on the regular simplex $S$, defined as
    \begin{equation}\label{Eq:ptri_def}
    p_{S}(x) := \frac{1}{N_d}\int_S \frac{1_{B_y}(x)}{U(B_y)}dU(y),
    \end{equation}
    where $B_y$ is the largest ball centered on $y$ which is contained in $S$, and $N_d$ is the normalization constant (we will prove below that indeed it only depends on $d$).  
    % Note that it only depends on $d$, as all the points in $S/4$, that is the shrunken simplex $1/4$), has a density that is greater by some $c_1(d)$, and the volume of $S/4$ is $4^{-d}\Vol(S)$.
    
    For any simplex $\triangle$, we define the density $p_{\triangle}$ as the pushforward of $p_S$ under the affine transformation $T$ sending $S$ to $T$. Note that its normalization is a constant that only depends on $d$. 
    
    % Hence, if for every $x \in \triangle$ we define $B_x$ to be the largest ball centered on $x$ which is contained in $\triangle$, and define the measure
    % \begin{equation}\label{Eq:ptri_def}
    % p_{\triangle}(x) := \int_\triangle \frac{1_{B_y}(x)}{U(B_y)}dU(y)
    % \end{equation}
    
    \begin{lemma}\label{Lem:BiasEst} Assume that  Let $M, M'$ be positive constants and that $\triangle=S$ is the regular simplex. Let $g: \triangle \to [-M'\|g\|_{L_1(U(\triangle))},M'\|g\|_{L_1(U(\triangle))}]$ be a convex $M$-Lipschitz function which is orthogonal in $L_2(U\triangle)$ to the affine functions.
    % such that $\|g - w_{g}\|_{L_2(\PP_{\triangle})} \geq c_1(M)\|g\|_{L_2(\PP_{\triangle})}$, where w_{g}:=
    Then, there exist positive constants such that:
    \begin{equation}\label{Eq:BiasEst}
        \int g d p_\triangle \asyd  C_1(M,M')\|g\|_{L_1(U(\triangle))}.
        %  - \int_{\triangle} w_{g} dU(x)
    \end{equation}
    In addition, for every affine function $w$,
    \[
    \int_{\triangle} w dp_{\triangle} = \int_{\triangle} w\, dU.
    \]
    Moreover,  there exists an efficient algorithm to compute $p_\triangle(x)$ for any $x \in \triangle$, and it holds that
    $$ C^{-d} \le \min_{x \in \triangle} p_\triangle(x)  \le \max_{x \in \triangle} p_\triangle(x) \le C^d,$$ 
for some absolute constant $C \geq 0$,
 % where $v_d$ is the volume of the unit ball in dimension $d$,
    \end{lemma}
    
    The idea behind the proof of this lemma is that for any point $x \in \triangle$ and a ball $B_x \subset \triangle$, the average $\bar g(x)$ of $g$ over $B_x$ is at least $g(x)$, with equality iff $g$ is affine on $B_x$. If $g$ is nonzero and orthogonal to affine functions, the averaged function $\bar g$ must have positive integral, and a compactness argument then yields a lower bound on $\frac{1}{\|g\|_1}\int \bar g$. The  proof is given in the Appendix.  

\subsubsection*{ACKNOWLEDGMENTS}
The first author is supported by NSF through award DMS-2031883 and from the Simons Foundation through Award 814639 for the Collaboration on the Theoretical Foundations of Deep Learning, as well as support from NSF through grant DMS-1953181. The second author was supported by the ERC under the European Union’s Horizon 2020 research and innovation programme (grant agreement No 770127). The first author would like to acknowledge Prof. Alexander Rakhlin, Prof. Mattias Reitzner and Dr. Nikita Zhivotevsky for the useful discussions. The authors  deeply acknowledge Prof. Shiri Artstein-Avidan for introducing the two authors to each other.

\bibliographystyle{abbrv}
    \bibliography{Bib}
% \newpage
\appendix
    \section{Definitions and Preliminaries}\label{App:Def}
    Here we collect definitions needed for the appendices below.
    
    % \begin{definition}\label{Def:Hausdorf}
    % The Hausdorff distance between to sets $A,B$ (w.r.t to the Euclidean distance)  is defined by
    % \[
    %     d_{H}(A,B) = \max\{\max_{x \in B} d(x,A),\max_{x \in A }d(x,B)\}.
    % \]
    % \end{definition}
    \begin{definition}\label{Def:Net}
    For a fixed $\eps \in (0,1)$, and a function class $\cF$ equipped with a probability measure $\QQ$, an $\eps$-net is a set that has the following property: For each $f \in \cF$ there exists an element in this set, denoted by $\Pi(f)$, such that $\|f-\Pi(f) \|_{\QQ} \leq \eps$.
    \end{definition}
    \begin{definition}\label{Def:NetB}
    We denote by $\cN(\eps,\cF,\QQ)$ the cardinality of the minimal $\eps$-net of $\cF$ (w.r.t to $L_2(\QQ)$).
    
    Also denote by $\cN_{[]}(\eps,\cF,\QQ)$ the cardinality of the minimal $\eps$-net with bracketing, which is defined as a set that has the following property: For each $f \in \cF$ there exists two elements $f_{-} \leq f \leq  f_{+}$ such that $\|f_{+}-f_{-} \|_{\QQ} \leq \eps$.
    \end{definition}
    % We use the terms $\eps$-entropy of $\cF$ w.r.t. $L_2(\PP)$ at scale $\eps$ (with bracketing) for the log-cardinality of the $\eps$-net (with bracketing). 
    Next, we recall the definition of $\PP$-Donsker and non $\PP$-Donsker classes for uniformly bounded $\cF$.
    \begin{definition}\label{Def:Donsker}
    $(\cF,\PP)$ is said to be $\PP$-Donsker if there exists $\alpha \in (0, 2)$ such that for all $\eps \in (0,1)$, we have $\log \cN_{[]}(\eps,\cF,\PP) = \Theta_{\PP,\cF}(\eps^{-\alpha})$, and non $\PP$-Donsker if the same holds with $\alpha \in (2, \infty)$.
    \end{definition}
    \begin{remark}\label{Rem:Donsker}
    It is shown in \citep{bronshtein1976varepsilon,gao2017entropy} that
    \[
        \log \cN_{[]}(\eps,\cF_{L}(\Omega) ,\PP) = \Theta_{d}((L/\eps)^{d/2})
    \]
    and
    \[
          \log \cN_{[]}(\eps,\cF^{\Gamma}(P) ,\PP) = \Theta_{d}(C(P)(\Gamma/\eps)^{d/2}).
    \]
    Therefore, when the dimension $d \geq 5$, both $\cF_{L}(\Omega)$ and $\cF^{\Gamma}(P)$ are non-Donsker classes.
    \end{remark}
    
    \paragraph{Basic notions regarding polytopes}
    A quick but thorough treatment of the basic theory is given in, e.g. \citep[\S 2.4]{schneider2014convex}. A set
    $P \subset \mathbb \R^{d}$ is called a polyhedral set if it is the intersection of a finite set of half-spaces, i.e., sets of the form $\{x \in \mathbb \R^{d}: x \cdot a \le c\}$ for some $a \in \mathbb \R^{d}$, $c \in \mathbb R$. A polyhedral set $P$ is called a polytope if it is bounded and has nonempty interior; equivalently, a set $P$ is a polytope if it is the convex hull of a finite set of points and has nonempty interior. 
    
    The affine hull of a set $S \subset \mathbb \R^d$ is defined as 
    $$\aff S = \bigcup_{k = 1}^\infty \{\sum_{i = 1}^k a_i x_i: x_i \in K, a_i \in \mathbb R\,|\,\sum_{i = 1}^k a_i = 1\},$$
    which is the minimal affine subspace of $\mathbb \R^{d}$ containing $S$. For a convex set $K$, we define its dimension to be the linear dimension of its affine hull.
    
    For any unit vector $u$ and any convex set $K$, the support set $F(K, u)$ is defined as 
    $$F(K, u) = \{x \in K: x \cdot u = \max_{y \in K} y \cdot u\}.$$
    (If $\max_{y \in K} y \cdot u = \infty$ then $F(K, u)$ is defined to be the empty set.) 
    
    Suppose $P$ is a polyhedral set. For any $u \in \mathbb{S}^{m - 1}$, $F(P, u)$ is a polyhedral set of smaller dimension than $K$. Any such $F(P, u)$ is called a face of $P$, and if $F(P, u)$ has dimension $m - 1$, it is called a facet of $P$. A polyhedral set $P$ which is neither empty nor the whole space $\mathbb \R^{d}$ has a finite and nonempty set of facets, and every face of $P$ is the intersection of some subset of the set of facets of $P$. If $P$ is a polytope, all of its faces, and in particular all of its facets, are bounded. A polytope is called simplicial if all of its facets are $(m - 1)$-dimensional simplices, which is to say, each facet $F$ of $P$ is the convex hull of precisely $m$ points in $\aff F$.
    \section{Proofs of Missing Parts}
    \subsection{Proof of Theorem \ref{Cor:Aproximation}}\label{App:One}
    % \subsection*{Notations and Preliminaries}
    % Finally, we introduce a classical lemma from convex geometry (cf. \citep{artstein2015asymptotic}).
    % \begin{lemma}\label{Lem:SimpleProb}
    % Let $\Omega \subset B_d$, and let $X_1,\ldots,X_n \underset{i.i.d}{\sim} U(\Omega)$. Then, with probability of at-least $1-Cn^{-2}$, we have that for each $x \in \Omega$, we can find $\Pi(x) \in \{X_1,\ldots,X_n\}$ such that $\|x-X_i\| \leq O_d((\log(n)n^{-1})^{1/d})$, and moreover that ***.
    % \end{lemma}
    % \begin{proof}[Proof of Theorem \ref{Cor:Aproximation}]
    Since the squared $L^2$-error scales quadratically with the function to be estimated, it suffices to prove the theorem for the class of $1$-Lipschitz functions. Since the range of a $1$-Lipschitz function on a domain of diameter at most $1$ is contained in an interval of length $1$, it is no loss to assume that the range of $f^*$ is contained in $[0, 1]$.
    
    The construction of a $k$-affine approximation to any convex $1$-Lipschitz function $f^*: \Omega \to [0, 1]$, uses a combination of two tools: the theory of random polytopes in convex sets, and empirical processes. %
    
    Fix a convex body $K \subset \mathbb \R^{d}$ and $n \ge d + 1$. The random polytope $K_n$ is defined to be the convex hull of $n$ random points $X_1, \ldots, X_n \sim U(K)$, where $U(\cdot)$ denotes the uniform distribution. It is well-known and easy to justify that $K_n$ is a simplicial polytope with probability $1$: Indeed, if $X_1, \ldots, X_n$ form a facet of $K_n$ then in particular they lie in the same affine hyperplane, and if $k \ge d + 1$, the probability that $X_k$ lies in the affine hull $H$ of $X_1, \ldots, X_n$ is $0$, since $K \cap H$ has volume $0$. For future use we note that with probability $1$, the projection of every facet of $K_n$ on the first $d - 1$ coordinates is a $(d - 1)$-dimensional simplex, by similar reasoning.
    % cf. (1.11) and the paragraph beneath
    % \begin{theorem}\label{Thm:rpolyvol} For any $d \geq 2$, $n \geq Cd^{d/2}$ and a convex set $K \subset \mathbb \R^{d}$, the following holds with probability of at least $1-n^{-2}$
    % $$  1 - \frac{\vol(K_n)}{\vol(K)} \le Cd n^{-\frac{2}{d + 1}},$$.
    % % where $C(d) = O(d)$.
    % \end{theorem}
    % \DrEli{the explicit $C(d)$ is known here, please add it and add a reference}
    
    For $s \in \{0, 1, \ldots, d - 1\}$ and $P$ a polytope, we let $f_s(P)$ denote the number of $s$-dimensional faces of $P$. The first result regarding random polytopes that we need appears in \citep[Corollary 3]{barany1989intrinsic}:
    
    \begin{theorem}\label{Thm:facets} Let $d \geq 1$,  $ 1 \leq s \leq d-1 $ and a convex body $K \subset \mathbb \R^{d}$. Then,  there exists $C(d,s) \leq C_1(d)$ such that 
    \[ 
    \mathbb E[f_s(K_n)] \le C(d,s) n^{\frac{d - 1}{d + 1}}.
    \]
    \end{theorem}
    We will also use the following classical estimates \cite{brunel2013adaptive,schutt2003polytopes} and references within.
    \begin{theorem}\label{Thm:RndPolyVol}
    Let $\Omega\subset B_d$ be a convex body, and let $Y_1,\ldots,Y_m \underset{i.i.d.}{\sim} U(\Omega)$. Then, $P_m = \conv(Y_1, \ldots, Y_m)$ is a simplicial polytope with probability $1$, and the following holds:
    \[
     \E U(P \setminus P_m) \lesssim d \cdot m^{-\frac{2}{d+1}}.
    \]
    Furthermore, if $\Omega \subset B_d$, then
    \[
        \E d_H(P,P_m) \lesssimd (\log(m)/m)^{\frac{1}{d}},
    \]
    where $d_{H}$ denotes the Hausdorff distance.
    % and
    % \[
    %  \E f_{d-1}(S_m) = O_d(C(P)\log(m)^{d-1}).
    % \]
    \end{theorem}

    % \DrEli{There are known upper bounds on $C(d,s)$ please state them}
    The other result that we need from empirical processes appears as Lemma \ref{Lem:Bro} in the main text. We now describe our construction. Given a $1$-Lipschitz function $f^*: \Omega \to [0, 1]$, define the convex body 
    $$K = \{(x, y): x \in \Omega, y \in [0, 2]\,|\, f^*(x) \le y\}.$$
    In other words, $K$ is the epigraph of the function $f^*$, intersected with the slab $\mathbb R^d \times [0, 2]$. Note that $\vol_{d - 1}(\Omega) \le \vol_d(K) \le 2 \vol_{d - 1}(\Omega)$, since $\mathrm{Im}\,f^* \subset [0, 1]$.
    
    Let $n = \lfloor k^{\frac{d + 2}{d}}\rfloor$, and consider the random polytope $K_n \subset K$. Let $\Omega_k$ be the projection of $K_n$ to $\mathbb R^d$, and define the function $f_k: \Omega_k \to [0, 2]$ by 
    $$f_k(x) = \min \{y \in \mathbb R: (x, y) \in K_n\},$$ 
    i.e., $f_k$ is the lower envelope of $K_n$. In particular, since $K_n \subset K$, $f_k$ lies above the graph of $f^*$. We would like to show that with positive probability, $f_k$ satisfies the properties in the statement of the theorem. We treat each property in turn.
    
    \paragraph{$f_k$ is $k$-simplicial with probability at least $9/10$:}
    Using Theorem \ref{Thm:facets}, and Markov's inequality $K_n$ has at most $10 C(d) n^{\frac{d}{d + 2}} = C'(d) k$ facets (recall that all facets of $K_n$ are simplices with probability $1$). Letting $\triangle_1, \ldots, \triangle_F$ be the bottom facets of $K_n$, and letting $\pi: \mathbb R^d \times \mathbb R \to \mathbb R^d$ be the projection onto the first factor, $\pi(\triangle_1), \ldots, \pi(\triangle_F)$ is a triangulation of $\Omega_k$ and for each $i = 1, \ldots, F$, $\left.f_k\right|_{\triangle_i}$ is affine, as its graph is simply $\triangle_i$.
    
    \paragraph{Bounding $
    \PP(\Omega \backslash \Omega_k)$ with probability at least $9/10$:}
    Since $\Omega_k$ is the projection of $K_n$ to $\mathbb R^d$, it is equivalently defined as $\conv(\pi(X_1), \ldots, \pi(X_n))$ where $X_1, \ldots, X_n$ are independently chosen from the uniform distribution on $K$, and $\pi$ is the projection onto the first $d$ coordinates as above. $\pi(X_i)$ is not uniformly distributed on $\Omega$, so we cannot apply Theorem \ref{Thm:RndPolyVol} (and Markov's inequality directly). Instead, we re-express $\pi(X_i)$ as a mixture of a uniform distribution and another distribution, and apply Theorem \ref{Thm:RndPolyVol} to the points which come from the uniform distribution.
    
    In more detail, note that we may write $K = K_1 \cup K_2$ where $K_1 = (\Omega \times [0, 1]) \cap \mathrm{epi}\,f$ and $K_2 = \Omega \times [1, 2]$, since $f \le 1$. Let $p = \frac{\vol(K_2)}{\vol(\Omega)} \ge \frac{1}{2}$. The uniform distribution from $K$ can be sampled from as follows: with probability $p$, sample uniformly from  $K_2$, and with probability $1 - p$ sample uniformly from $K_1$. Clearly, if $X$ is uniformly distributed from $K_2$ then $\pi(X)$ is uniformly distributed on $\Omega$. Hence, $\Omega_k$ can be constructed as follows: draw $M$ from the binomial distribution $B(n, p)$ with $n$ trials and success probability $p$, then sample $M$ points $X_1, \ldots, X_M$ uniformly from $\Omega$ and sample $k - M$ points $X_1', \ldots, X_{k - m}'$ from some other distribution on $\Omega$, which doesn't interest us; then set $\Omega_k = \conv(X_1, \ldots, X_M, X_1', \ldots, X_{k - M}')$. In particular, $\PP(\Omega\backslash \Omega_k) \ge \PP(\Omega \backslash \Omega_M)$, so it is sufficient to bound the RHS with high probability.
    
    By the usual tail bounds on the binomial distribution, $M \ge \frac{np}{2} \geq \frac{n}{4}$ with probability $1 - e^{-\Omega(n)}$.  Hence, by Theorem \ref{Thm:RndPolyVol} we obtain
    % the expected value of
    \begin{align*}
         \E \PP(\Omega \backslash \Omega_M) &\leq  \E \PP(\Omega \backslash \conv\{X_1,\ldots,X_{n/4}\}) + C(d)e^{-c(d)n} \lesssimd k^{-\frac{2(d+2)}{d(d+1)}},
    \end{align*}
    % $\frac{\vol(\Omega \backslash \Omega_M)}{\vol(\Omega_M)}$ is at most 
    % $$C(d) n^{\frac{d - 1}{d + 1}} \le C(d) k^{-\frac{d + 2}{d}\cdot \frac{2}{d + 1}}.$$
    and we obtain $\PP(\Omega \backslash \Omega_M) \lesssimd k^{-\frac{2(d+2)}{d(d+1)}}$ with probability at least $\frac{9}{10}$ by Markov's inequality.
    % so with probability at least $\frac{2}{3}$, $\frac{\vol(\Omega \backslash \Omega_M)}{\vol(\Omega_M)} \le 3 C(d) k^{-\frac{2}{d}}$.
    
    \paragraph{Bounding $\int (f - f_k)^2d\PP$ with probability at least $9/10$:}
    Finally, we wish to bound the $\mathbb L^2(\PP)$-norm of $f^* - f_k$. To do this, we use the same strategy, arguing that on average, $k$ of the points of $K_n$ can be thought of as drawn from the uniform distribution on a thin shell of width $k^{-\frac{2}{d}}$ lying above the graph of $f^*$, which automatically bounds the empirical $L^2$-norm $\int (f^* - f_k)^2\,d\mathbb P_n$ and hence the $L^2$-norm by Lemma \ref{Lem:Bro}.
    
    Now for the details. Set $\epsilon = k^{-\frac{2}{d}}$, and define 
    $$K_\epsilon = \{(x, y): x \in \mathbb R^d, y \in [0, 2]\,|\,f^*(x) \le y \le f^*(x) + \epsilon\},$$
    i.e., $K_\epsilon \subset K$ is just the strip of width $\epsilon$ lying above the graph of $f^*$. By Fubini, $K_\epsilon$ has volume $\epsilon \vol(\Omega) \ge \frac{\epsilon}{2 \vol(K)}$ , and if $X$ is uniformly distributed on $K_\epsilon$, $\pi(X)$ is uniformly distributed on $\Omega$. Hence, we can argue precisely as in the preceding: with probability $1 - e^{-\Omega(n)}$,
    $$L :=  |\{X_i: X_i \in K_\epsilon\}| \ge \frac{\epsilon n}{4} = \frac{k}{4}.$$
    Conditioning on $L$ for some $L \ge \frac{k}{4}$ and letting $X_1, \ldots, X_L$ be the points drawn from $K$ which lie in $K_\epsilon$ we have that $\pi(X_1), \ldots, \pi(X_L)$ are uniformly distributed on $\Omega$. Moreover, for any $i \in \{1, \ldots, n\}$, $X_i \in K_n$ and so it lies above the graph of $f_k$, but also $X_i \in K_\epsilon$ and so it lies below the graph of $f^* + \epsilon$. Combining these two facts yields 
    $$\forall 1 \leq i \leq L: \quad f_k(\pi(X_i)) \le (X_i)_{d + 1} \le f^*(\pi(X_i)) + \epsilon,$$
    where $(\cdot)_{d+1}$ denotes the $d+1 $ coordinate. Hence,
    \begin{equation}\label{Eq:bbb}
    \forall 1 \leq i \leq L: \quad f^*(\pi(X_i)) \le f_k(\pi(X_i)) \le f^*(\pi(X_i)) + Ck^{-2/d}.
    \end{equation}
    Thus, letting $\mathbb P_L = \frac{1}{L} \sum_{i = 1}^L \delta_{\pi(X_i)}$ denote the empirical measure on $\pi(X_1), \ldots, \pi(X_L)$, we obtain  
    $$\int_\Omega (f^* - f_k)^2\,d\mathbb P_L \le \frac{1}{L} \sum_{i = 1}^L \epsilon^2 = \epsilon^2 = k^{-\frac{4}{d}}.$$
    Since the $\pi(X_i)$ are drawn uniformly from $\Omega$, if we knew that $f_k$ were $1$-Lipschitz it would follow from Lemma \ref{Lem:Bro} that
    $$\int_\Omega (f^* - f_k)^2\,d\mathbb P \le k^{-\frac{4}{d}} + C(d) k^{-\frac{4}{d}} \lesssimd k^{-\frac{4}{d}},$$  with high probability.
    
    % Although $f_k$ is not $1$-Lipschitz, we know that $f^* \le f_k \le f^* + \epsilon$, and $f^* + \epsilon$ is $1$-Lipschitz, so we can apply Lemma \ref{Lem:Bro} to it. Hence we obtain
    
    % $$\int (f^k - f^*)^2\,d\mathbb P \le \int ((f^* + \epsilon) - f^*)^2\,d\mathbb P $$
    
    We do not know, however, that $f_k$ is $1$-Lipschitz. To get around this, define the function $\widehat{f}_k$ as the function on $\Omega_k$ whose graph is
    $
        \conv\{(\Pi(X_i),f^{*}(\Pi(X_i)))\}_{i=1}^{L}.
    $
    Unlike $f_k$, $\widehat f_k$ is necessarily $1$-Lipschitz since $f^*$ is (see, e.g., the argument in the paragraph below equations \eqref{Eq:ConvCons1}-\eqref{Eq:ConvCons3}), so by Lemma \ref{Lem:Bro}, it follows that 
    \[
    \int_\Omega (f^* - \widehat{f}_k)^2\,d\mathbb P \le  C_1k^{-\frac{4}{d}}
    \]
    with probability at least $1-C(d)\exp(-c(d)k)$. Also, by \eqref{Eq:bbb}, 
    \[
        \forall 1 \leq i \leq L: \quad  \widehat{f}_k(\pi(X_i)) \leq f_k(\pi(X_i)) \le  \widehat{f}_k(\pi(X_i)) + C_1k^{-2/d}.
    \]
    It easily follows by the definitions of $f_k$ and $\widehat f_k$ as convex hulls that on the domain $\Omega_{\Pi(X)}:= \conv\{(\Pi(X_i))\}_{i=1}^{L}$, we have
    \[
        \quad f^{*} \leq \widehat f_k \le f_k \leq \widehat{f}_k  + Ck^{-2/d}.
    \]
    Hence, we conclude that
    \begin{align*}
          \int_{\Omega_{\Pi(X)}}  (f_k - f^*)^2\, d\PP &\leq 2\int_{\Omega_{\Pi(X)}}  (\widehat f_k - f^*)^2\,d\PP + 2\int_{\Omega_{\Pi(X)}}  (f_k - \widehat f_k)^2\,d\PP \\
          &\le 2\int_{\Omega_{\Pi(X)}}  (\widehat f_k - f^*)^2\,d\PP + 2\| f_k - \widehat f_k\|_{L^\infty(\Pi(X))}^2 \\&\lesssim k^{-4/d}
    \end{align*}
    with high probability.
    % A result that is proven in \citep[Cor 1.]{dumbgen1996rates}\footnote{The high probability bounds do not appear directly in the corollary but follow from their proof.} that gives an estimate of the Hausdorff distance between $K_n$ to $K$ shows that it is at most $C_dk^{-\frac{(d+2)}{d(d+1)}}\log(k)^{\frac{1}{d+1}}$. Since $f^{*}$ is $1$-Lipschitz function, we obtain that
    % \[
    %     \| (f_k - f^{*})\mathbbm{1}_{\mathrm{supp}f_k}\|_{\infty} \leq Ck^{-\frac{d+2}{d(d+1)}}\log(k)^{\frac{1}{d+1}}.
    % \]
    % with probability of at least $1-k^{-2}$. 
    Now, using Theorem \ref{Thm:RndPolyVol} and Markov's inequality, we also know that
    \[
        U(\Omega \setminus \Omega_{\Pi(X)}) \lesssim dk^{-\frac{2(d+2)}{d(d+1)}}.
    \]
    with probability at least $\frac{19}{20}$. Conditioned on this event, and using the fact that $f_k$ is $L$-Lipschitz, and that $d_{H}(\Omega \setminus \Omega_{\Pi(X)}) \lesssimd (k/\log(k))^{-\frac{d+2}{d^2}}$, that 
    \[
        \int_{\Omega \setminus \Omega_{\Pi(X)}}(f_k - f^{*})^2 \lesssimd (k/\log(k))^{-\frac{2(d+2)}{d^2}} k^{-\frac{2(d+2)}{d(d+1)}} \ll C(d)k^{-\frac{4}{d}}
    \]
    where we used that\[
\frac{2(d+2)}{d^2}+\frac{2(d+2)}{d(d+1)} - \frac{4}{d}
\;=\; \frac{2(3d+2)}{d^2(d+1)} .
\]
    % Ck^{-\frac{2(d+2)}{d(d+1)}}\log(k)^{\frac{2}{d+1}}    
    \paragraph{Deriving the theorem} Since we have three events each of which hold with probability at least $9/10$, then the intersection of these events is not empty. Therefore, an $f_k$ satisfying all the desired properties exists, and the theorem follows.
    \begin{remark}
        When the domain $\Omega$ is a polytope with a number of facets that is bounded by $C_3(d)$, and the underlying convex function is bounded by $\Gamma$. We have to take a detour, by using a different argument, using the result of \cite{dwyer1988convex}, we know that
        \[
            \E U(P \setminus P_m) \lesssimd C(\Omega) k^{-\frac{d+2}{d}} \cdot \log(k)^{d-1},
        \]
        and use the uniform boundedness of the convex function, and by a $\Gamma$ on $\Omega \setminus \Omega_{\Pi(X)}$ rather than $ d_{H}(\Omega,\Omega_{\Pi(X)})$.
    \end{remark}

     \subsection{Proof of Lemma \ref{Lem:RandomSimp}}\label{App:RandomSimp}
    We start with the following easy lemma:
    \begin{lemma}\label{Lem:Standard}
    The following event holds with probability at least $1-n^{-3d}$:
    \begin{equation}
      \forall 1\leq  i \leq k(n) \  \text{ s.t. } \PP(\triangle_i) \geq C_3d\log(n)/n: \ \     2^{-1}\PP(\triangle_i )\leq \PP_n(\triangle_i) \leq 2\PP(\triangle_i).
    \end{equation}
    \end{lemma}
    \begin{proof}
    The lemma follows for the fact that $n\cdot \PP_n(S) \sim Bin(n,\PP(S))$, along with the concentration inequality (cf. \citep{boucheron2013concentration}) for binomial random variables: for all $\eps \in (0,1)$,
    \[
        \Pr\left(\left|\frac{\PP_n(S)}{\PP(S)} - 1\right| \leq \eps\right) \leq 2\exp(-c\min\{\PP(S),1-\PP(S)\}n\eps^2).
    \]
    By taking $\eps = 1/2$, and choosing $C$ to be large enough, we conclude that for any particular $\triangle_i$, 
    \[
        \PP(\triangle_i) \geq C_3d\log(n)/n: \ \     2^{-1}\PP(\triangle_i )\leq \PP_n(\triangle_i) \leq 2\PP(\triangle_i)
    \]
    with probability at least $1 - n^{-(3d + 1)}$. Taking the union bound over all $k(n)$ simplices, the claim follows.
    \end{proof}
    
    The main step is the following lemma, which shows that for any given simplex $\triangle_i$, if we draw $Cd \log n$ points from the uniform distribution on $\triangle_i$ for sufficiently large $C$, then there exists some subset $S$ of these points whose convex hull $P$ covers almost all of the simplex and can also be triangulated by a polylogarithmic number of simplices whose vertices lie in $S$.
    
    \begin{lemma}\label{Lem:SimpDecomp}
    Let $S \subset \R^d$ be a simplex, and $m \gtrsim 
    d\log(n)$. Let $Y_1,\ldots,Y_m \sim \PP_{S}$. Then, with probability at least $1-n^{-3d}$ there exists a set $\cA$ of simplices contained in $S$ with disjoint interiors of cardinality $|\cA| \leq C_d\log(m)^{d-1}$ such that
    \[
     \PP_{S}(S \setminus \bigcup \cA) \lesssimd \frac{\log(n)\log(m)^{d-1}}{m}.
    \]
    \end{lemma}
    \begin{proof}
    For each $s \in \{0, 1, \ldots, d - 1\}$ and $P$ a polytope, we let $f_s(P)$ denote the number of $s$-dimensional faces of $P$. We need the following result, which was first proven in \citep{dwyer1988convex}; for more details see the recent paper \cite{reitzner2019convex}. 
     
    \begin{theorem}
    Let $S \subset \R^d$ be a simplex, and let $Y_1,\ldots,Y_m \sim \PP_{S}$. Then, $S_m = \conv(Y_1, \ldots, Y_m)$ is a simplicial polytope with probability $1$, and the following holds:
    \[
     \E \PP_{S}(S \setminus S_m) = O_d(m^{-1}\log(m)^{d-1}),
    \]
    and
    \[
     \E f_{d-1}(S_m) = O_d(\log(m)^{d-1}).
    \]
    \end{theorem}
    This theorem does not give us what we need directly, since it treats only expectation while we require high-probability bounds. (To the best of our knowledge, sub-Gaussian concentration bounds are not known for the random variables $f_{d-1}(P_m),\PP(S \setminus S_m)$ when $S$ is a simplex, cf. \citep{vu2005sharp}.) This necessitates using a partitioning strategy.
    % , but rather only when $K$ is smooth,
    We divide our $Y_1,\ldots,Y_{m}$ into $C_1 d\log(n)$ blocks, for $C_1$ to be chosen later, each with $m(n) := \frac{m}{C_1 d\log(n)}$ samples drawn uniformly from $\triangle$. Let $P_1, \ldots, P_B$ be the convex hulls of the points in each block, each of which are independent realizations of the random polytope $S_{m(n)}$. For each $P_i$, Markov's inequality and a union bound yield that with probability at least $\frac{1}{3}$,
    \begin{align}\label{Eq:vol_bd}
        \PP_{S}(S \setminus P_i) &\leq 3\cdot\E\PP_{S}(S \setminus P_i) \lesssimd  m(n)^{-1}\log(m(n))^{d-1}  \\&\lesssimd \frac{\log(m)^{d-1}\log(n)}{n},
    \end{align}
    and 
    \begin{equation}\label{Eq:fac_bd}
        f_{d-1}(P_i) \leq  3\E f_{d-1}(S_{m(n)}) \lesssimd  \log(m(n))^{d-1} \lesssimd \log(m)^{d-1}.
    \end{equation}
    Since there are $C_1 d\log(n)$ independent $P_i$, at least one of them  will satisfy these conditions with probability $1 - \left(\frac{2}{3}\right)^{C_1 d \log n}$, and we may choose $C_1$ so that this is at least $1 - n^{-3d}$.
    
    Conditioned on the existence of $P_i$ satisfying \eqref{Eq:vol_bd} and \eqref{Eq:fac_bd}, we take one such $P_i$ and triangulate it by picking any point among the original $Y_1, \ldots, Y_m$ lying in the interior of $P_i$ and connecting it to each of the $(d - 1)$-simplices making up the boundary of $P_i$. The set $\cA$ is simply the set of $d$-simplices in this triangulation.
    \end{proof}
    
    Now, to obtain Lemma \ref{Lem:RandomSimp}, we condition on the event of Lemma \ref{Lem:Standard} and apply Lemma \ref{Lem:SimpDecomp} to each $\triangle_i$ such that $\PP(\triangle_i) \geq Cd\log(n)/n$, with the $Y_1, \ldots, Y_m$ taken to be the points of $X_{n + 1}, \ldots, X_{2n}$ drawn from $\PP$ which fall inside of $\triangle_i$. Using the fact that $\PP_n(\triangle_i) \geq 0.5\PP(\triangle_i)$, we see that $m \ge Cd \log n$ for each $\triangle_i$, so Lemma \ref{Lem:SimpDecomp} is in fact applicable. In addition, the bounds on the cardinality of $\mathcal S_X^i$ and on the volume of $\triangle_i$ left uncovered by the simplices in $\mathcal S_X^i$ follow immediately by substituting $c \PP(\triangle_i)$ for $m$ in the conclusions of Lemma \ref{Lem:SimpDecomp}. For $i$ such that $\PP(\triangle_i) \le Cd\log(n)/n$, we take $\mathcal S_X^i$ to be the empty set.
    
    \subsection{Proofs of Lemmas \ref{Lem:VerMSE} and \ref{Lem:LoneNorm} }\label{App:VerMSE}
    In several places, we will use a high-probability estimator for the mean of a random variable presented in \citep{devroye2016sub}:
     \begin{lemma}\label{Lem:HighPEst}
    Let $\delta \in (0,1)$ and let $Z_1,\ldots,Z_k$ be i.i.d. samples from a distribution on $\R$ with finite variance $\sigma_{Z}^2$. There exists an estimator $\widehat{f}_{\delta}: \mathbb R^k \to \mathbb R$ with a runtime of $O(k)$, such that with probability at least $1 - \delta$,
    \[
         (\widehat{f}_{\delta}(Z_1, \ldots, Z_k) - \E Z)^2 \lesssim \frac{\sigma_Z^2 \cdot \log(2/\delta)}{k}.
    \]
    \end{lemma}
    
    First, we construct the estimator for the $L_1$ norm of a convex function $g$ defined on a convex body $K$, which is the content of Lemma \ref{Lem:LoneNorm}. Then, we show how to ``upgrade'' this estimator to an estimator of the $L_2$-norm in the special case that $K$ is a simplex. 
    The final step will be to estimate the $L_2$ norm of $g$ under the assumptions of Lemma \ref{Lem:VerMSE}, by using Lemma \ref{Lem:LoneNorm} and the claim of Lemma \ref{Lem:VerMSE} will follow. 
    
    % **Also, up an affine transformation (with a determinant one) that does not affect the norm, we may assume that $\triangle = S_{r(\triangle)}$, where $S_{r(\triangle)}$ is the regular simplex with with radius $r(\triangle) = \vol(\triangle)^{1/d}$. For each $\delta \in (0,1)$, we will use the notation  $\triangle^{\delta}:= \vol(\triangle)^{1/d}(1-\delta)\triangle$, and  for the the $L_{p}(dx)$ measure.
    % Clearly, our goal is to find a sharp bound on the $L_2$ norm of $g$.    Namely we will estimate  $\|g\mathbbm{1}_{\triangle^{s}}\|_1$. In the second step, we estimate the $L_1$ norm of on the remaining shell $\triangle \setminus \triangle^{s} $.  Using these estimates, we will obtain the above estimate on the $L_1$ norm of $g$. In the final step, we obtain the claim of Lemma \ref{Lem:VerMSE}.
    
    \subsubsection{Proof of Lemma \ref{Lem:LoneNorm}}\label{sssec:inner_simplex}
    We only prove this Lemma for $K$ a simplex, which is what we require for our algorithm. The proof for a general $K$ can be done in similar fashion, by placing $K$ in John position (besides the computational aspects).
    % $\|g\|_1$
    % First, recall that aim to aim to estimate the $L_2$ norm of $g$ is orthogonal to the affine functions, i.e. that $g \perp w_{*}$, where $w^*$ is the best linear approximation to $g_{*}$ (w.r.t to $L_2(\triangle)$). 
    % We may assume that the samples are generated from the above $g$.
    
    Let $S$ be the regular simplex inscribed in the unit ball $B_d$, and for each $t \in [0,1]$, denote by $S_{t} := (1-t)S$. We will use the following geometric facts, which can be extracted from the statements and proofs of \citep[Lemmas 2.6-2.7]{gao2017entropy}.
    
    \begin{lemma}\label{Lem:paring}
     Let $g: S \to \R$ be a convex function, and let $\|g\|_1 = \int_S |g|$ be its $L^1$-norm. We have:
     \begin{itemize}
         \item $- \|g\|_1 \lesssimd g(x) $ on all of $x \in S$.
         \item For each $\delta \in (0,1)$, the restricted $\left.g\right|_{S_{\delta}}$ is is $O_d(\delta^{-d}\|g\|_1)$-Lipschitz and $O_d(\delta^{-d} \|g\|_1)$ uniformly bounded.
     \end{itemize}
    %  \[
    
    %  \]
    %  where $C_d$ is a constant that only depends on $d$.
    %  and that $g_{|_{\triangle^{\delta}}}$ is .
    %  where $\nabla $ denotes the gradient (or some member in the sub-differential).
    \end{lemma}
    We immediately obtain the following corollary:
    %  {\citep{gao2017entropy}}
    \begin{corollary}\label{Cor:Gao}
    For any $a \in (0,1)$, there exists a constant $\delta := \delta(a)$ such that $g$ restricted to $S_{\delta}$ is uniformly bounded by $O_d(\|g\|_1)$; moreover, letting $g_- = \min(g, 0)$, we have
    \[\int_{S \setminus S_{\delta}} |g_{-}|\, dU(x) \leq a\|g\|_{1}.\]
    \end{corollary}

    %The observation that leads to this measure is based on the following idea: a convex function $g$ with zero mean and barycenter must be negative near the barycenter of $\triangle$ and positive near the boundary. One can therefore hope that the integral of $g$ in a neighborhood of the boundary gives a lower bound on the $L^1$-norm of $g$. 
    
    Using Lemma \ref{Lem:LoneNorm}, and the above results, we can estimate the $L_1$ norm $g:\triangle \to [-1,1]$. Letting $T$ be the unique affine transformation such that $T \triangle = S$, we define the shrunken simplex $\triangle_{\delta}$ by the $\triangle_{\delta}:= T^{-1}(T\triangle)_{\delta (l,d)}$, where $a=1/10$ and $\delta (a,d)$ is defined in Corollary \ref{Cor:Gao}. The proof involves analysis of several cases.
    
    \paragraph{Case 1:} $\|g\mathbbm{1}_{\triangle \setminus \triangle_{\delta}}\|_{1} \geq \frac{1}{2}\|g\|_{1}$, i.e. most of the $L_1$-norm of $g$ comes from the shell $\triangle \setminus \triangle_{\delta}$. Using Corollary \ref{Cor:Gao}, we know that
    \[
        \|g\|_{1} \geq \int_{\triangle \setminus \triangle_{\delta}} gdU(x) = \int_{\triangle \setminus \triangle_{\delta}} g^{+}dU(x) + \int_{\triangle \setminus \triangle_{\delta}} g^{-}dU(x) \geq (3/20) \cdot \|g\|_{1}.     
    \]
    Therefore it is enough to estimate the mean (scaled by $U(\triangle \setminus \triangle_{\delta})$)  of the r.v. $g(X)$ where $X \sim U(\triangle \setminus \triangle_{\delta})$, which can be done using the samples that fall in $\triangle \setminus \triangle_{\delta}$.
    %  (recall that we may assume that $\PP_n(\triangle \setminus \triangle_{\delta}) \geq 0.5\PP(\triangle \setminus \triangle_{\delta})-Cd \log(n)/n$)
    By Lemma \ref{Lem:HighPEst} above, we conclude that using these samples we have an estimator $\widehat{f}_{(1)}$ such that with probability at least $1-\delta$,
    \begin{align}
        \left|\widehat{f}_{(1)} - \int_{\triangle \setminus \triangle_{\delta}}g\,dU \right|^2 &\leq U(\triangle \setminus \triangle_{\delta})^2\frac{Cd(\sigma^2+\| g\mathbbm{1}_{\triangle \setminus \triangle_{\delta}}\|_2^2)\log(2/\delta)}{U(\triangle \setminus \triangle_{\delta})m} \nonumber \\
        &\lesssimd  \frac{(\sigma^2+\|g\|_2^2)\log(2/\delta)}{m},
        \label{l1_case1}
    \end{align}
    where we used that fact that $U(\triangle \setminus \triangle_{\delta}) \geq c_d$.
    
    \paragraph{Case 2:} If we are not in Case 1, we must have $\|g\mathbbm{1}_{\triangle_{\delta}}\|_{1} \geq \frac{1}{2}\|g\|_{1}$, i.e., most of the $L^1$-norm of $g$ comes from the inner simplex $\triangle_\delta$. Decompose $g = w_{g} + (g - w_{g})$, where $w_g = \argmin_{w \text{ affine }} \|g - w\|_{L_2(U(\triangle_{\delta}))}$ is the $L^2(\triangle_\delta)$-projection of $g$ onto the space of affine functions. Note that by orthogonality, we have \begin{equation}\label{Eq:orth_bd}
        \max(\|w_g\|_{L_2(U(\triangle_{\delta}))}, 
        \|g - w_g\|_{L_2(U(\triangle_{\delta}))}) \le \|g\|_{L_2(U(\triangle_\delta))}
    \end{equation}
    by orthogonality, while by Lemma \ref{Lem:paring}, we have 
    \begin{equation}\label{eq:l1_l2}
        \|g\|_{L_1(U(\triangle_\delta))} \le \|g\|_{L_2(U(\triangle_\delta))}\| \le C_d \|g\|_{L_1(U(\triangle_\delta))}.
    \end{equation}
    
    By the triangle inequality, we must have either $\|w_g\mathbbm{1}_{\triangle_{\delta}}\|_{1} \geq  \frac{1}{4}\|g\|_{1}$ or $\|(g - w_{g})\mathbbm{1}_{\triangle_{\delta}}\|_1 \geq \frac{1}{4}\|g\|_{1}$; we analyze each case below.
    
    \paragraph{Case 2a:} First, suppose $\|w_g\mathbbm{1}_{\triangle_{\delta}}\|_{1} \geq  \frac{1}{4}\|g\|_{1}$. Using half of the samples that fall into $\triangle_{\delta}$, we may apply Lemma \ref{Lem:LinearReg}, giving us an affine $\widehat{w}_g$ such that with probability at least  $1-\delta$,
    \[
        \|\widehat{w}_g - w_{g}\|_{2}^2 \leq C_d(\sigma^2+\|g\|_2^2)\frac{\log(2/\delta)}{m}.
        % \leq \frac{C_1d\log(n)}{\PP(\triangle)n}. 
    \]
    Writing $\bar{f}_{(2a)}:= \|\widehat{w}_{g}\|_{1}$, we conclude that
    \begin{equation}\label{l1_case2a}
        \frac{1}{4}\|g\|_{1} - C_d
        (\sigma+\|g\|_2) \sqrt{\frac{\log(2/\delta)}{m}}  \leq \bar{f}_{(2a)} \leq \|g\|_{1} + C_d(\sigma+\|g\|_2) \sqrt{\frac{\log(2/\delta)}{m}}.
    \end{equation}
    Note that the right-hand inequality does not require the assumption $\|w_g\mathbbm{1}_{\triangle_{\delta}}\|_{1} \geq  \frac{1}{4}\|g\|_{1}$.
    
    \paragraph{Case 2b:} Now suppose $\|(g - w_{g})\mathbbm{1}_{\triangle_{\delta}}\|_1 \geq \frac{1}{4}\|g\|_{1}$. To estimate $\|(g - w_{g})\mathbbm{1}_{\triangle_{\delta}}\|_1$, we will use our Lemma \ref{Lem:BiasEst}. Note that by the definition of $\triangle_{\delta}$, we may assume by Lemma \ref{Cor:Gao} that $\max \{M',M\} \leq C(d)$ (in Lemma \ref{Lem:BiasEst}). Therefore, we conclude that
    \[
     c_1(d)\|g - w_{g}\|_{1} \leq  \int_{\triangle_\delta} (g - w_g) \cdot p_{\triangle_{\delta}} \leq C_1(d)\|g - w_{g}\|_{1}.
    \]
    % \int_{\triangle_{\delta}} w_g dU_{\triangle_{\delta}}  
    By our assumption $\|(g - w_{g})\mathbbm{1}_{\triangle_{\delta}}\|_1 \geq \frac{1}{4}\|g\|_{1}$; we also have 
    $$\|(g - w_{g})\mathbbm{1}_{\triangle_{\delta}}\|_1 \le \|g\|_1 
    \le \|\widehat w_g - w_g\|_1 + \|(g - w_{g})\mathbbm{1}_{\triangle_{\delta}}\|_1$$
    So it follows that
    \[
        % - C_d (\sigma+\|g\|_2) \sqrt{\frac{\log(2/\delta)}{m}} 
    \|g\|_1 \asyd \int_{\triangle_\delta} (g-\widehat{w}_g) dp_{\triangle_{\delta}}. 
        % \leq  C_3(d)\|g-w_g\|_{1} + C_d
        % |\sigma+\|g\|_2|\sqrt{\frac{\log(2/\delta)}{m}} 
        % \leq  C_3(d)\|g\|_{1} + C_d
        % |\sigma+\|g\|_2|\sqrt{\frac{\log(2/\delta)}{m}}.
    \]
    Therefore, it is enough to estimate $\int(g -\widehat{w}_g) dp_{\triangle_{\delta}}$, using the second half of the samples that fall into $\triangle_{\delta}$.

    To do this, we simulate sampling from $p_{\triangle_{\delta}}$ given samples from $U(\triangle_{\delta})$ and their corresponding noisy samples of $g - \widehat{w}_g$. The idea is simply to use rejection sampling \citep{devroye1986nonuniform}: given a single sample $X \sim U(\triangle_{\delta})$, we keep it with probability $O_d(p_{\triangle_{\delta}}(x))$. Conditioned on keeping the sample, $X$ is distributed according to $p_{\triangle_{\delta}}$. If we are given $m/2$ i.i.d. samples from $U(\triangle_{\delta})$, then with probability $1 -e^{-c_d' m}$ the random number of samples $N$ we obtain from $p_{\triangle}$ by this method is at least $c(\alpha_d)\cdot m \geq c_1(d)m$, and conditioned on $N$, these samples are i.i.d. $p_{\triangle_{\delta}}$. We condition on this event going forward. Now, using these $N$ samples, Lemma \ref{Lem:HighPEst} gives an estimator $\bar{f}_{(2b)}$ such that 
    \begin{equation}\label{l1_case2b}
        \left|\bar{f}_{(2b)} - \int (g - \widehat{w}_g)\cdot p_{\triangle_\delta}\right| \leq C_d (\sigma + \|g\|_{2}) \sqrt{\frac{\log(2/\delta)}{m}}.
    \end{equation}
    with probability at least $1-\max\{\delta,e^{-cm}\}$.  
    
    Note that each of $\bar{f}_{(1)},\bar{f}_{(2a)},\bar{f}_{(2b)}$ is bounded from above by $C_d\|g\|_{1} + C_d(\sigma + \|g\|_{2}) \sqrt{\frac{\log(2/\delta)}{m}}$, irrespective of whether we are in the case for which the estimator was designed; this follows from \eqref{l1_case1}, \eqref{l1_case2a}, \eqref{l1_case2b}, respectively. %To see this, $\bar{f}_{(1)}$ is bounded by $C_d\|g\|_{1}$, since $U(\triangle \setminus \triangle_{\delta}) \geq c_d$. Next, for $\bar{f}_{(2)}$ it follows from fact that
    % \[
    %     \|w_g\|_{L_1(\PP_{\triangle_{\delta}})} \leq \|w_g\|_{L_2(\PP_{\triangle_{\delta}})} \leq \|g\|_{L_2(\PP_{\triangle_{\delta}})} \leq C(d)\|g\|_{2} \leq C_1(d)\|g\|_{1}
    % \]
    % where we used the fact that $\|p_{\triangle_{\delta}}\|_{\infty} \leq C_d$ and that $\|g\mathbbm{1}_{\triangle_{\delta}}\|_{\infty} \leq C_2(d)\|g\|_{1}$. Next, for $\bar{f}_{(3)}$ the fact that $p_{\triangle_{\delta}} \leq \alpha_d \cdot \PP$ gives the claim. 
    
    Finally, by using $\bar{f}_{(1)},\bar{f}_{(2a)},\bar{f}_{(2b)}$, we conclude that with probability $1-3\max\{\delta,e^{-cm}\}$ at least
    \[
        c_1(d)\|g\|_{1} -  C_d(\sigma + \|g\|_{2}) \sqrt{\frac{\log(2/\delta)}{m}} \leq \bar{f}_{(1)} + \bar{f}_{(2a)} + \bar{f}_{(2b)} \leq C_1(d)\|g\|_{1} +  C_d|\sigma + \|g\|_{2}| \sqrt{\frac{\log(2/\delta)}{m}},
    \]
    and the claim follows.
    \subsection{Proof of Lemma \ref{Lem:BiasEst}}
    % Recall that we assume w.l.o.g. that $\|f^{*}\|_{\infty} \leq L$ (where in the proof we assume that $L=1$) 
    % , also recall that the affine transformation $T$ such that $T(\triangle) = \triangle_0$ has all eigenvalues bounded below by some constant depending only on $d$.x
    % Let $\triangle_0$ be the regular simplex in $\R^d$ and consider the transformation $g \circ T^{-1}: \triangle_0 \to \mathbb R$ is a convex function that is supported on $\triangle_0$ and normalize $g \circ T^{-1}$ to be a $1$-norm convex function on the regular simplex. i.e. we also assume that $\|g\|_{2}:= \int_{\triangle_0} g(x)dU(x)$ equals to 1, where $U$ denotes the uniform measure on the regular simplex.
    Recall that it suffices to show that
    %  Note that in order to prove  \eqref{Eq:BiasEst}, it is sufficient for a convex $g:\triangle \to [-1,1]$ such that the following holds:  
    \[
            c_1(M,M') \le \int_{S} g(x)p_{S}(x)dx \le C_1(M,M').
    \]
    % In order to simply the notation we assume that $\triangle = \triangle_0$ and fix some $l \in (0,1)$ that will be defined later.
    Note that the function $g$ is convex and in particular subharmonic, i.e., for any ball $B_x$ with center $x$ contained in $S$ we have 
    $$\frac{1}{U(B_x)} \int_{B_x} g(z)dU(z) \ge g(x),$$
    where $U$ denotes the uniform measure on the regular simplex $S$.
    $g$ is non-affine and hence strictly subharmonic (as convex harmonic functions are affine), so there exists some $x$ such that for any ball $B_x \subset S$ centered on $x$, the above inequality is strict, since subharmonicity is a local property. As $g$ is convex and in particular continuous, the inequality is strict on some open set of positive measure. We obtain that for a \emph{non}-affine convex function that
    \begin{equation}\label{Eq:gp_pos}
    \begin{aligned}
    \int_S g(x) p_{S}(x)\,dU(x) &= \int_S \int_S g(x) 1_{B_y}(x)\,dU(y)\,dU(x) \\&= \int_S \left(\frac{1}{U(B_y)}\int_{B_y} g(x)\,dx\right)\,dU(y) > \int_S g(y)dU(y) = 0,
    \end{aligned}
    \end{equation}
    i.e. we showed that for a subharmonic $g$ that
    $
        \int_S g(x) p_{S}(x)\,dx > 0.
        %  - \int_\triangle g(y)dU(x) 
    $
    % where we used the fact  $\int g = 0$ is affine i.e. also harmonic, it is clear that for such $p_{\triangle}$, $\int wp_{\triangle} = \int_{\triangle} wdU(x)$. 
    
    Now, we show why  \eqref{Eq:gp_pos} actually implies the lower bound of  \eqref{Eq:BiasEst}, which is certainly not obvious a priori. However, it follows from a standard compactness argument. The set $\mathcal C$  of convex $M$-bounded, $M'$-Lipschitz functions with norm $1$ that is orthogonal to the affine functions is closed in $L^\infty(S)$, and also equicontinuous due to the Lipschitz condition. Hence, by the Arzela-Ascoli theorem it is compact in $L^\infty(S)$, and we conclude that 
    $$A = \left\{ \int_{S}g(x)p_{S}(x)dx: g \in \mathcal C\right \}$$ 
    is compact; but \eqref{Eq:gp_pos} implies that $A \subset (0, \infty)$, which finally implies the existence of $c(M,M',d) > 0$ such that $S \subset [c(M,M',d), \infty)$. As for the upper bound in \eqref{Eq:BiasEst}, it follows immediately from the boundedness of $p_{S}$, which we prove below.
    % It remains to explain how to choose the balls $B_x$ in the definition of $p$ so that $p$ is bounded and efficiently computable in the sense claimed. It turns out that the choice which perhaps first comes to mind is satisfactory: for every $x \in \triangle$.
    
     We claim that in this case \eqref{Eq:ptri_def} can be evaluated analytically as a function of $x$, though the formulas are sufficiently complicated that this is best left to a computer algebra system. Indeed, we note that $y \in S$ contributes to the integral at $x$ if and only if $x$ is closer to $y$ than $y$ is to the boundary of $S$. The regular simplex can be divided into $d + 1$ congruent cells $C_1, \ldots, C_{d+1}$ such that the points in $C_i$ are closer to the $i$-th facet of the simplex than to any other facet (in fact, $C_i$ is simply the convex hull of the barycenter of $S$ and the $i$th facet); for any $y \in C_i$, $x \in B_y$ if and only if $x$ is closer to $y$ than $y$ is to the hyperplane $H_i$ containing $C_i$. But the locus of points equidistant from a fixed point $x$ and a hyperplane is the higher-dimensional analog of an elliptic paraboloid, for which it's easy to write down an explicit equation. Letting $P_{i, x}$ be the set of points on $x$'s side of the paraboloid (namely, those closer to $x$ than to $H_i$), we obtain
    $$p_{S}(x) \propto \sum_{i = 1}^{d+1} \int_{C_i \cap P_{i, x}} \frac{dy}{v_d \cdot d(y, H_i)^d},$$
     where $v_d$ is the volume of the unit ball in dimension $d$.
     Each region of integration $C_i \cap P_{i, x}$ is defined by several linear inequalities and a single quadratic inequality, and the integrand can be written simply as $\frac{1}{y_i^d}$ in an appropriate coordinate system. It is thus clear that the integral can be evaluated analytically, as claimed.
    
    Finally, we need to show that $p_{S}(x)$ is bounded above by $C^d$. It follows by simple observation, choose $y$ to be the center of the smallest ball $y$, i.e. $1_y(x)/U(B_y)$ is maximal, then we know by triangle inequality that $x$ is $2r(B_y)$ close the boundary point that the ball $B_y$ hit.  Hence, the total volume that $x$ can be integrated is $2^dU(B_y)$, now by integrating over all possible radii the proof is complete. Note that a lower bound holds as well, to see this choose $y = x$, and a ball with radius $r(B_x)/2$, and on each point on this ball $1_{x 
    \in B_y}$ by integration, we obtain that $P_{S}(x) \geq C^{-d}$, and the proof is complete.

    \subsubsection{Proof of Lemma \ref{Lem:VerMSE} }\label{sssec:simplex_l2}
    % {Estimating \texorpdfstring{$\|g\|_2$}{||g||₂}}
    Recall that we are given a convex $L$-Lipschitz function $g$ satisfying $\|g\|_{\infty} \leq L$; by homogeneity, we may assume $L = 1$. Our goal in this subsection is to estimate $\|g\|_2$ up to polylogarithmic factors given an estimate of $\|g\|_1$, where $g:\triangle \to [-1,1]$. This part requires the additional assumption that $\|g\|_2 \ge \frac{ Cd^{1/2}\log n}{n^{1/2}}$.
    
    %  which implies in particular that that $\|g\|_1 \ge \frac{\|g\|_2^2}{\|g\|_\infty} \ge C \frac{d(\log n)^{2}}{n}$.
    
    %For this part it is more convenient to adopt a different normalization; recalling that $g$ is bounded by $1$ and that we are only need consider $g$ having $L^2$-norm at least $\sqrt{\frac{\log n}{n}}$ -- implying that $\|g\|_1 \ge \frac{\|g\|_2^2}{\|g\|_\infty} \ge \sqrt{\frac{\log n}{n}}$ -- we multiply $g$ by a constant, so that $1 = \|g\|_{1} \leq \|g\|_{2} \le \|g\|_\infty \le C\sqrt{n/\log(n)}$. 

    For this section, we will need the following classical result about the floating body of a simplex \citep{baranylarman1988,schutt1990convex}.
    \begin{lemma}\label{Lem:SimpFloat}
    For a simplex $S$ and $\eps \in (0,1)$. let $S_{\eps}$ be its $\eps$-convex floating body, defined as 
    \[
       S_{\eps} := \bigcap \{K: K \subset S \text{ convex },\vol(S\setminus K) \leq \eps\vol(S)\},
    \]
    and let $S(\epsilon) = S \backslash S_\eps$ be the so-called wet part of $S$. Then $\vol(S(\epsilon)) \le C_d \eps \log(\eps^{-1})^{d - 1} \vol(S)$.
    \end{lemma}
    
    We also note that for any particular $\eps$ and $x \in S$ one can check in polynomial time whether $x \in S(\eps)$: indeed, letting 
    \begin{align*}
    H_{x, u}^+ &= \{y \in \mathbb R^n: \langle y, u\rangle \ge \langle x, u\rangle\} \\
    H_{x, u} &= \partial H_{x, u}^+ = \{y \in \mathbb R^n: \langle y, u\rangle = \langle x, u\rangle\},
    \end{align*}
    the function $u \mapsto U(S \cap H_{x, u}^+)$ is smooth on $S^{d - 1}$ outside of the closed, lower-dimensional subset $A$ where $H_{x, u}$ is not in general position with respect to some face of $u$, and, moreover, is given by an analytic expression in each of the connected components $C_i$ of $S^{d - 1} \backslash A$. It can thus be determined algorithmically whether $\min_i \inf_{x \in C_i}  U(S \cap H_{x, u}^+) \le \epsilon$, i.e., whether $x \in S(\epsilon)$.
    
     Let $v = \mathbb P(S)$, and let $i_{min} = \min(\lfloor\log_2(C_d \log(n)^{d + 1} v^{-1}) \rfloor, 0)$; note that since $\|g\|_{1} \geq v \ge \frac{\log(n)^d}{n}$, $|i_{min}| \le C \log n$. Set $V = g^{-1}((-\infty, 2^{i_{min}}])$, and for $i = i_{min}, i_{min} + 1, \ldots, 0$, set $U_i = g^{-1}((2^i, 1])$. Note that $V$ is convex, while each $U_i$, $i \ge 0$, is the complement of a convex subset of $S$.
    
    We will use the following lemma:
    \begin{lemma}\label{Lem:LEVELSET} For $g$ and $V, U_i$ as defined above, at least one of the following alternatives holds:
    %*** ELI I changed it to $(d+1)/2$ *** ]
    \begin{enumerate}
        \item
        $
        c_d\log(n)^{-d+1/2}\|g\|_{2} \leq v^{-\frac{1}{2}} \|g\|_1 \leq \|g\|_2.
       $
        \item  There exists $i_0 \in [i_{min}, 0]$ such that $2^{-i_0} \ge C_d(\log n)^{d-1} \frac{\|g\|_1}{\PP(S)}$ and
        \begin{equation}\label{Eq:i0}
        c\log(n)^{-1/2} \|g\|_{2} \leq \PP(U_{i_0})^{-1/2}\int_{U_{i_0}}g\,d\PP.
        \end{equation}
        % \|g\|_2/(2\sqrt{\log(n)}) \leq 2^{i_0}\PP(T_{i_0})^{1/2} \leq \|g\|_{2}.
    \end{enumerate}
    \end{lemma}
    
    The proof of this lemma appears at the end of this subsection.
    
    If alternative (1) of the lemma holds, the $L_1$-norm of $g$ is only a polylogarithmic factor away from the $L_2$-norm (up to normalizing by the measure of $S$, which is known to us). Therefore, we may use the $L_1$-estimator of the previous subsection and estimate the $L_2$ norm of $g$, up to a larger polylogarithmic factor, as we will see below.
    
    % For this purpose, we will first prove the following lemma:
    % \begin{lemma}\label{Lem:FromShToLvl}
    % Suppose alternative (2) of Lemma \ref{Lem:LEVELSET} holds, and let $i_0 \in [i_{min}, i_{max}]$ satisfying \eqref{Eq:i0}. Then, the following holds:
    %  \[
        % \frac{1}{2\log(n)^{1/2}} \|g\|_{2} \leq \PP(U_{i_0})^{-1/2}\int_{U_{i_0}}g\,d\PP_{U_{i_0}}
    %  \]
    %  Also, for all $i\in [i_{min}, i_{max}]$ we have that
    %  \[
    %     \PP(U_i)^{1/2}\int_{U_i}g\,d\PP_{U_{i}} \leq \|g\|_{2}.
    %  \]
    % \end{lemma}
    We must therefore consider what happens when alternative (2) of Lemma \ref{Lem:LEVELSET} holds. If we could estimate the integral of $g$ over $U_{i_0}$, we'd be done, but neither the index $i_0$ nor the set $U_{i_0}$ are given to us. So we make use of the fact that each such $U_{i_0}$, being the complement of a convex subset of $\triangle$, is contained in the wet part $S(\mathbb P(U_i))$, which has volume at most $C_d\log(n)^{d-1}\PP(U_i)$ by Lemma \ref{Lem:SimpFloat}. We will show in the next lemma that this replacement costs us a $C_d \log(n)^{d/2}$ factor in the worst case.
    % two lemmas. For this first one, we define for each $i \in \mathbbm{N}$,  
    % $
    %     U_i := \{ g \geq 2^{i}\}. 
    % $
    % \begin{lemma}
    % \end{lemma}
    
    More precisely, let $\epsilon_j = 2^{-2j}$, and let $S(\eps_j)$ be the corresponding wet part of $S$, as defined in Lemma \ref{Lem:SimpFloat}. Then we have the following:
    
    \begin{lemma}\label{Lem:Float} With $g$ as above, we have
     \begin{equation*}
        \max_{j \in [i_{min}, 0]} \PP(S(\eps_j))^{-1/2}\int_{S(\eps_j)}g d\PP \leq \|g\|_{2}, 
    \end{equation*}
    and moreover, if alternative (2) of Lemma \ref{Lem:SimpFloat} holds, then there exists $j$ such that $\PP(S(\eps_j)) \ge \frac{Cd\log n}{n}$ and 
    $$c_d\log(n)^{-d/2}\|g\|_{2}\leq \PP(S(\eps_j))^{-1/2}\ \int_{S(\eps_j)}g d\PP.$$
    \end{lemma}
    Using the last lemma, we can construct an estimator for $\|g\|_2$ that is at most a polylogarithmic factor away from the true value, whether we are in case (1) or case (2) of Lemma \ref{Lem:LEVELSET}. 
    
    Indeed, note that we if alternative (2) holds, we have $\mathbb P(S_{\epsilon_j}) \ge \frac{Cd\log n}{n}$ and hence, as in Section 2, we can assume by a union bound that the number of sample points falling in $S(\epsilon_j)$ is proportional to $\PP(S(\epsilon_j))$. Hence, for each $j$ the estimation of $\int_{S(\epsilon_j)} g\,d\PP_{S(\epsilon_j)} =\PP(S(\epsilon_j))^{-1}\int_{S(\epsilon_j)} g\,d\PP $ can be done in a similar fashion as in \S \ref{sssec:inner_simplex}, with an additive deviation that is proportional to $\sqrt{\frac{\log(2/\delta)}{n\PP(S(\eps_j))}}$. However, since we need to estimate $\PP(S(\epsilon_j))^{-\frac{1}{2}} \int_{S(\epsilon_j)} g\,d\PP$, we can multiply it by $\sqrt{\PP(S(\epsilon_j))}$, and obtain the correct deviation of $O(\sqrt{\log(2/\delta)/(\PP(S)n)})$ (as usual, we work on the event $\PP(S(\eps_j))/2 \leq  \PP_n(S(\eps_j))$, which holds with probability at least $1-n^{-2d}$).
    
    We conclude that the maximum over $j \in [-c\log(n) \leq i_{min}, 0]$, estimators of the means of the random variables $\PP(S(\epsilon_j))^{\frac{1}{2}}\int_{S(\eps_j)}g\, d\PP_{S(\eps\triangle_i)}$ and the $L_1$ estimator of the above sub-sub section give the claim.  
    
    % Therefore, we conclude that by taking the max over $ 1 \leq j \leq C\log(n)$ over the estimators of $$ our $L_1$ , estiators of the last 

    % \[
    %      \|g\|_2/(4\log(n)) \leq \|g\mathbbm{1}_{L_0}\|_{1} \leq \|g\mathbbm{1}_{L_0} \|_{2} \leq   C_1(d)\log(n)^{\frac{d-1}{2}}\|g\mathbbm{1}_{L_0}\|_{1} \leq C_1(d)\log(n)^{\frac{d-1}{2}}\|g\|_{1}.
    % \]
    % Now, 
    % and we know that
    % \[
    %     L_i \subset S \setminus S^{\ep\triangle_i}
    % \]
    % and we will show that
    % \begin{equation}\label{Eq:Float}
    %      \|g\|_{2} \leq A_i:= C_d\log(n)^{(d-1)/2+1}\sqrt{\vol(S \setminus S^{\ep\triangle_i})}\int_{S \setminus S^{\ep\triangle_i}}gd\PP  \leq \|g\|_2\log(n)^{d}.
    % \end{equation}
    % If \eqref{Eq:Float} holds, and the first case holds, then we can derive an estimate $\|g\|_2$ by considering
    % \[
    %   \max_{ 0 \leq i \leq \log(n)/2}A_i .
    % \]
    % In the remaining of this \S, we prove  \eqref{Eq:Float}.
    % \paragraph{Proof of  \eqref{Eq:Float}} First, note that it is enough to show that 
    % \[
    %     \sqrt{\vol(S \setminus S^{\ep\triangle_i})}\int_{S \setminus S^{\ep\triangle_i}}g^{+}d\PP
    % \]
    % to see this, recall that by Lemma ** that
    % $
    %     g \geq -C_d\|g\|_{1}
    % $
    % and by Markov's inequality we know that $\vol(U_i) $ ***.
    % Here, we provide the deferred proofs of the lemmas.
    \begin{proof}[Proof of Lemma \ref{Lem:LEVELSET}]
    Let $g_- = \min(g, 0)$, $g_+ = \max(g, 0)$, so that $\|g\|_2^2 = \|g_-\|^2 + \|g_+\|^2$. 
    
    We claim that alternative (1) holds if $\|g_-\|_2^2 \ge \frac{1}{2}\|g\|_2^2$. Indeed, by Lemma \ref{Lem:paring}, $g_- \ge -C_d v^{-1} \|g\|_1$, which immediately yields 
    $$\int_S g_-^2\,d\mathbb P \le -C_d v^{-1} \|g\|_1 \cdot \int g_-\,d\mathbb P \le -C_d v^{-1} \|g\|^2_1$$
    i.e.,
    $$v^{-1/2}\|g\|_1 \ge c_d \|g_-\|_2 \ge c_d' \|g\|_2.$$ 
    Note that by Jensen's inequality we have that $v^{-1/2}\|g\|_{1} \leq \|g\|_{2}$.
    Otherwise, we have $\|g_+\|^2 \ge \frac{1}{2} \|g\|_2^2$. Let $T_i = U_i \backslash U_{i + 1} = g^{-1}((2^i, 2^{i + 1}])$. We have 
    $$\frac{1}{2} \|g\|_2^2 \le \|g_+\|_2^2 \le  \sum_{i = -\infty}^{0} 2^{2(i + 2))} \mathbb P(T_i).$$
    By our assumption $\|g\|_2^2 \ge C\frac{\log n}{n}$ and the fact that $v \le 1$, the terms in the sum with $i \le i_{min} = \log \left(C\frac{\log n}{n}\right) + 2$ cannot contribute more than half of the sum, so we have .
    $$\frac{1}{4} \|g\|_2^2 \le \|g_+\|_2^2 \le  \sum_{i = i_{min}}^{0} 2^{2(i + 2))} \mathbb P(T_i).$$
    
    Hence there exists $i_0 \in [i_{min}, 0]$ such that 
    $$\frac{\|g\|_2^2}{4\log n} \le 2^{2(i_0 + 2))} \mathbb P(T_{i_0}) \le 4 \min_{x \in U_i} g(x)^2 \cdot \mathbb P(U_i) \le 4\mathbb P(U_{i_0})^{-1} \left(\int_{U_i} g\,d\mathbb P\right)^2,$$
    or $$c\log(n)^{-\frac{1}{2}} \|g\|_2 \le \mathbb P(U_{i_0})^{-\frac{1}{2}} \int_{U_{i_0}} g\,d\mathbb P.$$
    
    We consider two cases: either $2^{-i} \ge C (\log n)^{d-1}  v^{-1} \|g\|_1$, or  $2^{-i} \le C (\log n)^{d-1}  v^{-1} \|g\|_1$. The first case leads immediately to alternative (2), while in the second case we have 
    \begin{align*}
        c \log(n)^{-\frac{1}{2}} \|g\|_2 \le \mathbb P(U_{i_0})^{-\frac{1}{2}} \int_{U_{i_0}} g\,d\mathbb P \le 4 C (\log n)^{d-1}  v^{-1} \|g\|_1 \cdot \mathbb P(U_{i_0})^{\frac{1}{2}} \le C (\log n)^{d-1} \|g\|_1 v^{-\frac{1}{2}},
    \end{align*}
    which is another instance of alternative (1).
    
    As for the right-hand inequality in alternative (1), this is simply Cauchy-Schwarz: $\left(\int |g|\,d\mathbb P\right)^2 \le \int g^2\,d\mathbb P \cdot v$.
    \end{proof}
    % \begin{proof}[Proof of Lemma \ref{Lem:FromShToLvl}]The first statement follows immediately from Lemma \ref{Lem:LEVELSET}, using the fact that $T_{i_0} \subset U_{i_0}$ and $g \ge 2^{i_0}$ on $U_{i_0}$.
    
    % The second statement is simply Cauchy-Schwarz: we have 
    % \begin{align*}\int_{U_i} g\,d\mathbb P_{U_i} \le \left(\int_{U_i} g^2\,d\mathbb P_{U_i}\right)^{\frac{1}{2}} \left(\int_{U_i} 1\,d\mathbb P_{U_i}\right)^{\frac{1}{2}} \le \|g\|_2 \cdot \mathbb P(U_i)^{\frac{1}{2}}.
    % \end{align*}
    
    % \end{proof}
    
    \begin{proof}[Proof of Lemma \ref{Lem:Float}]
    The first inequality is again Cauchy-Schwarz: for any subset $A$ of $S$, we have 
    \begin{align*}\int_{A} g\,d\mathbb P \le \left(\int_{A} g^2\,d\mathbb P\right)^{\frac{1}{2}} \left(\int_{A} 1\,d\mathbb P\right)^{\frac{1}{2}} \le \|g\|_2 \cdot \mathbb P(A)^{\frac{1}{2}}.
    \end{align*}
    
    As for the second statement, first note that since $\|g\|_{\infty} \leq 1$ and we have 
    $$c\log(n)^{-1/2} \|g\|_{2} \leq \PP(U_{i_0})^{-1/2}\int_{U_{i_0}}g\,d\PP \le \PP(U_{i_0})^{\frac{1}{2}}.$$
    Let $j= \lceil \log \PP(U_{i_0})\rceil \ge i_{min}$, $\epsilon_j = 2^j$, so that $U_{i_0} \subset S(\epsilon_j)$ and 
    $$\PP(S_{\epsilon_j}) \le C \PP(U_{i_0}) \log(\PP(U_{i_0})^{-1})^{d - 1} \le C \PP(U_{i_0}) (\log n)^{d - 1}.$$ 
    Recalling again that by \eqref{Lem:paring}, $g \ge -C \|g\|_1 \PP(S)^{-1}$, we have
    \begin{align*}\int_{S(\epsilon_j)} g\,d\mathbb P - 2^{-1}\int_{U_{i_0}} g\,d\mathbb P &\geq 2^{-1}\int_{U_{i_0}} g\,d\mathbb P + \int_{S(\epsilon_j)\backslash U_{i_0}} g\,d\mathbb P \\
    &\ge \PP(U_{i_0}) 2^{i_0} - \PP(S(\epsilon_j)) \cdot C \|g\|_1 \PP(S)^{-1} \\
    &\ge \PP(U_{i_0}) \left(2^{i_0} - C_d(\log n)^{d - 1} \|g\|_1 \PP(S)^{-1}\right) \\
    &\ge c\cdot \PP(U_{i_0})\cdot 2^{i_0} > 0,
    % c_1 \int_{U_{i_0}} g\,d\mathbb P,
    % \geq c\log(n)^{-1/2}\|g\|_{2}\PP(U_{i_0})^{1/2},
    \end{align*}
    where we used our assumption on $i_0$ in the last line. Therefore, by the last two inequalities
    $$\mathbb P(S(\epsilon_j))^{-\frac{1}{2}} \int_{S(\epsilon_j)} g\,d\mathbb P \ge c(\log n)^{-(d-1)/2} \mathbb P(U_{i_0})^{-\frac{1}{2}} \int_{U_{i_0}} g\,d\mathbb P \ge c(\log n)^{-d/2} \|g\|_2,$$
    as claimed. Finally, note that by the assumptions of $\|g\|_2^2 \ge  \frac{Cd(\log n)^2}{n}$ and $\|g\|_{\infty} \leq 1$, we obtain that
    \begin{equation*}
     \PP(S(\eps_j)) \geq \PP(U_{i_0}) \ge (c\log(n)^{-\frac{1}{2}} \|g\|_2/\sqrt{\|g\|_{\infty}})^2 \ge C_1d\frac{\log n}{n}.
    \end{equation*}
    \end{proof}
    \section{Sketch of Proof of Theorem \ref{Thm:ConvexFunc2}}\label{Sec:Cor}
    The modifications of Algorithm \ref{alg:one} to work in this setting are minimal: we simply need to replace $L$ by $\Gamma$, and replace  \eqref{Eq:ConvCons2} with
    \begin{align}%\label{Eq:ConvCons}
    \forall (i, j) \in [|\cS|] \times [n]: \qquad & |y_{i,j}| \le \Gamma 
    \end{align}
    \begin{equation*}
    \begin{aligned}
    &\forall \  1\leq i \leq |\cS| \ \  \frac{1}{n}\sum_{j=1}^{n}(f(Z_{i,j})-\widehat{w}_i^{\top}(Z_{i,j},1))^2 \leq \widehat{l}_i^2 + \Gamma\sqrt{\frac{Cd\log(n)}{n}}
    \\& \forall (i,j) \in [|\cS|] \times [n] \quad  |f(Z_{i,j})| \leq \Gamma 
    \\&   \forall (i_1,j_1),(i_2,j_2) \in [|\cS|] \times [n] \quad f(Z_{i_2,j_2}) \geq \nabla f(Z_{i_1,j_1})^{\top}(Z_{i_2,j_2}-Z_{i_1,j_1}).
    \end{aligned}
    \end{equation*}
    
    For the correctness proof, we need some additional modifications. First, we replace Lemma \ref{Lem:Bro} with a similar bound in the $\Gamma$-bounded setting. The following lemma is based on the $L_4$ entropy bound of \citep[Thm 1.1]{gao2017entropy} and the peeling device \citep[Ch. 5]{van2000empirical}; it appears explicitly in \citep{han2016multivariate}):
    \begin{lemma}\label{Lem:Gao}
     Let $d \geq 5$, $m \geq C^{d}$ and $\QQ$ be a uniform measure on a convex polytope $P' \subset B_d$ and $Z_1,\ldots,Z_m \sim \QQ$. Then, the following holds uniformly for all $f,g \in \cF^{\Gamma}(P')$ 
    \[
        2^{-1}\int_{P'} (f-g)^2 d\QQ - C(P')\Gamma^2m^{-\frac{4}{d}} \leq \int (f-g)^2 d\QQ_m \leq 2\int_{P'} (f-g)^2 d\QQ +  C(P')\Gamma^2m^{-\frac{4}{d}},
    \]
    with probability at least $1 - C_1(P')\exp(-c_1(P')\sqrt{m})$.
    \end{lemma}
    Note that differently from Lemma \ref{Lem:Bro}, the constant before $m^{-4/d}$ depends on the domain $P'$, and this dependence cannot be removed. Since $\mathcal F^\Gamma(P')$ has finite $L^2$-entropy for every $\epsilon$, it is in particular compact in $L^2(P')$, which means that the proof of Lemma \ref{Lem:VerMSE} in \S \ref{App:VerMSE} works for this class of functions as well.
    
    The proof of Theorem \ref{Cor:Aproximation} also goes through for this case, by replacing Lemma \ref{Lem:Bro} by Lemma \ref{Lem:Gao}. The precise statement we obtain is the following:
    
    \begin{theorem}\label{Cor:Aproximation2}
     Let $P \subset B_d$ be a convex polytope, $f \in \cF^{\Gamma}(P)$, and some integer $k \geq (Cd)^{d/2}$, for some large enough $C \geq 0$, there exists a convex set $P_k \subset P$ and a $k$-simplicial convex function $f_k: P_k \to \R$ such that 
    \begin{equation*}\label{Eq:support2}
    \PP(P \setminus P_{k}) \leq C(P)k^{-\frac{d+2}{d}}\log(k)^{d-1}.
    \end{equation*}
    and 
    \begin{equation*}\label{Eq:Error2}
        \int_{P_k}(f_k-f)^2d\PP \leq \Gamma^2 \cdot C(P)k^{-\frac{4}{d}} 
    \end{equation*}
    % and 
    % \[
    %     \|f_k-f^{*}\|_{\infty} \leq O_d(Lk^{-\frac{2(d+2)}{(d+1)}}).
    % \]
    %Furthermore,  we also have that
    %    (1-c_d\log(k)^{d-1}k^{-\frac{d+2}{d}})
    %  and the Hausdorff distance $d_{H}(\Omega  ,\Omega_k)$ is bounded by $ O_d(k^{-\frac{d+2}{d}\frac{1}{d}})$.
    \end{theorem}
    
    The remaining lemmas and arguments in the proof of Theorem \ref{Thm:ConvexFunc}, can easily be seen to apply in the setting of $\Gamma$-bounded regression under polytopal support $P$.

    \section{Simplified version of our estimator}\label{App:SimpleEstimator}
    
    Like the estimator for our original problem, the simplified version of our estimator is based on the existence of a simplicial approximation $\widehat f_{k(n)}: \Omega_{k(n)} \to [0, 1]$ to the unknown convex function $f^*$ (Theorem \ref{Thm:ConvexFunc}). Here we demonstrate how to recover $f^*$ to within the desired accuracy if we are given the simplicial structure of $f_{k(n)}$, i.e., the set $\Omega_{k(n)}$ and the decomposition $\bigcup_{i = 1}^{k(n)} \triangle_i$ of $\Omega_{k(n)}$ into simplices such that $ f_{k(n)}|_{\triangle_i}$ is affine for each $i$. In this case the performance of our algorithm is rather better: it runs in time $O_d(n^{O(1)})$ rather than $n^{O(d)}$, and is minimax optimal up to a constant that depends on $d$. We can also slightly weaken the assumptions: it is no longer required that the variance $\sigma^2$ of the noise be given.

     We will use the following classical estimator \citep[Thm 11.3]{gyorfi2002distribution}; it is quoted here with an improved bound which is proven in \citep[Theorem A]{mourtada2021distribution}:
    \begin{lemma}{\label{Lem:LinearRegE}}
     Let $m \geq d+1$, $d \geq 1$ and $\QQ$ be a probability measure that is supported on some $\Omega' \subset \R^d$. Consider the regression model $
        W= f^{*}(Z) + \xi$, 
     where $f^{*}$ is $L$-Lipschitz and $\|f^{*}\|_{\infty} \leq L$, and $Z_1,\ldots,Z_m \underset{i.i.d.}{\sim} \QQ$. Then, the exists an estimator $\widehat{f}_{R}$ that has an input of $\{(Z_i,W_i)\}_{i=1}^{m}$ and runtime of $O_d(n)$ and outputs a function such that
     \[
         \E \int (\widehat{f}_{R}(x) - f^{*}(x))^2d\QQ(x) \leq \frac{Cd(\sigma+ L)^2}{m} + \inf_{ w \in \R^{d+1}}\int (w^{\top}(x,1)-f^{*}(x))^2d\QQ(x).
    \]
    \end{lemma}
    Note that this estimator is \textit{distribution-free}: it works irrespective of the structure of $\mathbb Q$, nor does it require that $\mathbb Q$ be known.
    
    % Also, we will state our second approximation theorem:
    % \begin{theorem}\label{Thm:AproxGen}
    %  Let $\Omega \subset B_d$ be a convex set, $f \in \cF_{L}(\Omega)$, and some integer $k \geq (Cd)^{d/2}$, for some large enough $C \geq 0$, there exists a convex set $\Omega_k \subset \Omega$ and a $k$-simplicial convex function $f_k: \Omega_k \to \R$ such that 
    % \begin{equation}\label{Eq:supportG}
    % \frac{\PP(\Omega \setminus \Omega_{k})}{\PP(\Omega)} \leq Cdk^{-\frac{2(d+2)}{d(d+1)}}.
    % \end{equation}
    % and 
    % \begin{equation}\label{Eq:ErrorG}
    %     \int_{\Omega_k}(f_k-f)^2d\PP \leq O_d(L^2k^{-\frac{4}{d}}\log(k)^{\frac{2}{d+1}}).
    % \end{equation}
    % % and 
    % % \[
    % %     \|f_k-f^{*}\|_{\infty} \leq O_d(Lk^{-\frac{2(d+2)}{(d+1)}}).
    % % \]
    % %Furthermore,  we also have that
    % %    (1-c_d\log(k)^{d-1}k^{-\frac{d+2}{d}})
    % %  and the Hausdorff distance $d_{H}(\Omega  ,\Omega_k)$ is bounded by $ O_d(k^{-\frac{d+2}{d}\frac{1}{d}})$.
    % \end{theorem}

    The first step of the simplified algorithm is estimating $\left.f^*\right|_{\triangle_i}$ on each $\triangle_{i} \subset \Omega_{k(n)}$ ($1 \leq i \leq k(n)$) with the estimator $\widehat{f}_{R}$ defined in Lemma \ref{Lem:LinearRegE} (with respect to the probability measure $\PP(\cdot |\triangle_i)$) with the input of the data points in $\cD$ that lie in $\triangle_i$. We obtain independent regressors $\widehat{f}_1,\ldots,\widehat{f}_{k(n)}$ such that
    \begin{equation}\label{Eq:LinearRegKnown}
        \E \int_{\triangle_i} (\widehat{f}_i(x) - f^{*}(x))^2 \frac{d\PP}{\PP(\triangle_i)} \leq \inf_{w \in \R^{d+1}}\int_{\triangle_i}(w^{\top}(x,1)-f^{*}(x))^2 \frac{d\PP}{\PP(\triangle_i)}  + \E \min\{\frac{Cd}{\PP_n(\triangle_i)n},1\},
    \end{equation}
    where  the $\min \{\cdot, 1\}$ part follows from the fact that when we have less than $Cd$ points, we can always set $\widehat{f}_i$ to be the zero function.
    
    % Therefore, 
    % \begin{equation*}
    %     \E \int_{\triangle_i}(\widehat{w}_i^{\top}(x,1) - f^{*})^2 d\PP \leq \inf_{w \in \R^{d+1}}\int_{\triangle_i}(w^{\top}(x,1)-f^{*}(x))^2 d\PP  + \frac{C_1d}{n}.
    % \end{equation}
    Now, we define the function $f'(x) := \sum_{i=1}^{k(n)}\widehat{f}_i(x)\mathbbm{1}_{x \in \triangle_i}$, and by multiplying the last equation by $\PP(\triangle_i)$ for each $1 \leq i \leq k(n)$ and taking a sum over $i$, we obtain that  
    \begin{equation}\label{Eq:11}
    \begin{aligned}
         \E \int_{\Omega_{k(n)}} (f'- f^{*})^2 d\PP &\leq
        %  \sum_{i=1}^{k(n)} \inf_{w_i \in \R^{d+1}}\int_{\triangle_i}(w_i^{\top}(x,1)-f^{*})^2d\PP
         \sum_{i=1}^{k(n)} \inf_{w_i \in \R^{d+1}}\int_{\triangle_i}(w_i^{\top}(x,1)-f^{*})^2d\PP + \E \sum_{i=1}^{k(n)} \min\{\frac{Cd\cdot \PP(\triangle_i)}{n \cdot \PP_n(\triangle_i)},\PP(\triangle_i)\}
         \\& \leq \int_{\Omega_{k(n)}}(f_{k(n)}-f^{*})^2 d\PP + C_1dk(n) \cdot n^{-1}  =  O_d(n^{-\frac{4}{d+4}}), 
        %  \\& \leq \int_{\Omega_{k(n)}}(f_{k(n)}-f^{*})^2 d\PP  + O_d(n^{-\frac{4}{d+4}})
        %   =O_d(n^{-\frac{4}{d+4}}),
    \end{aligned}
    \end{equation}
    where in the first equation, we used the the fact that $n\cdot 
    \PP_n(\triangle_i) \sim  Bin(n,\PP(\triangle_i))$ (for completeness, see Lemma \ref{Lem:Standard}), and in the last inequality we used  \eqref{Eq:boundOne}. 
    % and our assumption of $C(P) = O_d(1)$.
     Next, recall that Theorem \ref{Cor:Aproximation} implies that
    \[
      \PP(\Omega \setminus \Omega_{k(n)}) \leq C(d)k(n)^{-\frac{d+2}{d}} \leq O_d(n^{-\frac{d+2}{d+4}}). 
    \]
    Therefore, if we consider the (not necessarily convex) function
    $
        \tilde{f} = f'\mathbbm{1}_{\Omega_{k(n)}} + \mathbbm{1}_{\Omega \setminus \Omega_{k(n)}},
    $
    we obtain that
    \begin{align*}
        \E \int_{\Omega} (f'- f^{*})^2 d\PP &= \E \int_{\Omega \setminus \Omega_{k(n)}} (f'- f^{*})^2 d\PP+ \E \int_{\Omega_{k(n)}} (f'- f^{*})^2 d\PP \leq O_d(n^{-\frac{d+2}{d+4}}+ n^{-\frac{4}{d+4}} )  \\&\leq O_d(n^{-\frac{4}{d+4}}).   
    \end{align*}
    % Hence, if we consider the function $\tilde{f}$
    % and therefore by the concentration of $M \sim Bin(n,\PP(\Omega \setminus \Omega_{k(n)}))$, we conclude by the last equation
    % \begin{align*}
    % \E \int_{\Omega \setminus \Omega_{k(n)}} (f_{\Omega_{k(n)}^{c}} - f^{*})^2 d\PP &\leq C_1\PP(\Omega \setminus \Omega_{k(n)})\E[M^{-\frac{2}{d}}] \leq C_2\PP(\Omega \setminus \Omega_{k(n)})(n\cdot \PP(\Omega \setminus \Omega_{k(n)}))^{-\frac{2}{d}}   \\&\leq O_d(n^{-\frac{2(d+2)}{(d+4)(d+1)}} n^{-\frac{2}{d}+\frac{4(d+2)}{d(d+4)(d+1)}}) = O_d(n^{-(\frac{4}{d+4}+\frac{6}{(1 + d)(4 + d)})}),
    % \end{align*}
    % and  \eqref{Eq:Boundar} follows. Hence, we conclude by Eqs. \eqref{Eq:11} and \eqref{Eq:Boundar}  that the $\bar{f}_n = f'\mathbbm{1}_{\Omega_{k(n)}}+f_{\Omega_{k(n)}^{c}}\mathbbm{1}_{\Omega \setminus \Omega_{k(n)}}$ satisfies
    % \[
    %     \E \int_{\Omega}(\bar{f}_n- f^{*})^2d\PP \leq O_d(n^{-\frac{4}{d+4}}).
    % \]
    
    Thus, $\tilde{f}$ is a minimax optimal \emph{improper} estimator. To obtain a proper estimator, we simply need to replace $\tilde f$ by $MP(\tilde{f})$, where $MP$ is the procedure defined in Appendix \ref{App:MP}. 
    
    It remains only to point out that the runtime of this estimator is of order $O_d(n^{O(1)})$.  Indeed, the procedure $MP$ is essentially a convex LSE on $n$ points, which can be formulated as a quadratic programming problem with $O(n^2)$ constraints, and hence can be computed in $O_d(n^{O(1)})$ time \citep{seijo2011nonparametric}. In addition, the runtime of the other estimator we use, namely the estimator of Lemma \ref{Lem:LinearRegE}, is linear in the number of inputs.
    %%%%%%%%%%%%%%%%%%%%%%%%%%%%%%%%%%%%%%%%%%%%%%%%%%%%%%%
    \section{From an Improper to a Proper Estimator}\label{App:MP}
    The following procedure, which we named $MP$, is classical and we give its description and prove its correctness here for completeness. However, note that we only give a proof for optimality in expectation; high-probability bounds can be obtained using standard concentration inequalities.
    
    The procedure $MP$ is defined as follows: given an improper estimator $\tilde f$, draw $X_{1}',\ldots,X_{k(n)}' \underset{i.i.d.}{\sim} \PP$, and apply the convex LSE with the input $\{(X_{i}',\tilde{f}(X_{i}'))\}_{i=1}^{k(n)}$, yielding a function $\widehat f_1$.
    % , denoted by  $\widehat{f}_{1}$.
    % \[
    %     % \begin{equation}
    % \E \int_{\Omega} (\widehat{f}_1 - f^{*})^2 d\PP =  O_d(n^{-\frac{4}{d+4}}).
    % \]
    We remark that the convex LSE is only unique on the convex hull of the data-points $X_{1}',\ldots,X_{k(n)}'$, and not on the entire domain $\Omega$ \citep{seijo2011nonparametric}, so we will show that any solution $\widehat f_1$ of the convex LSE is optimal.
    First off, we have
    % , note that by linearity of expectation and Eqs. \eqref{Eq:11} \eqref{Eq:Boundary} we have
    \begin{equation}\label{Eq:FixedDesign}
        \E \int_{\Omega} (\tilde{f} - f^{*})^2 d\PP'_{k(n)} =   \E \int_{\Omega} (\tilde{f} - f^{*})^2 d\PP 
        % \leq O_d(n^{-\frac{4}{d+4}}\log(n)^{d+1}).
    \end{equation}
    
    Also recall the classical observation that for $\widehat{f}_1$ that is defined above, we know that \[(\widehat{f}_1(X_{1}'),\ldots,\widehat{f}_1(X_{k(n)}'))\]  is precisely the projection of $(\tilde{f}(X_{1}'),\ldots,\tilde{f}(X_{k(n)}'))$ on the \emph{convex set}
    \[
        \cF_{k(n)}:=\{(f(X_{1}'),\ldots,f(X_{k(n)}')): f \in \cF_1(\Omega)\} \subset \mathbb R^{k(n)},
    \]
    cf. \citep{chatterjee2014new}. Now, the function $\Pi_{\cF_{k(n)}}$ sending a point to its projection onto $\cF_{k(n)}$, like any projection to a convex set, is a $1$-Lipschitz function, i.e.,
    \[
        \|\Pi_{\cF_{k(n)}}(x) - \Pi_{\cF_{k(n)}}(y) \| \leq \|x-y\| \quad \forall x,y \in \R^{k(n)}.
    \]
    We also know that $(\widehat{f}_1(X_i'))_{i = 1}^{k(n)} = \Pi_{\cF_{k(n)}}(\tilde f)$ and  $\Pi_{\cF_{k(n)}}((f^{*}(X_i'))_{i = 1}^{k(n)}) = (f^{*}(X_i'))_{i = 1}^{k(n)}$; substituting in the preceding equation, we therefore obtain
    % $
    %     (\tilde{f}'(X_{n+1}),\ldots,\tilde{f}'(X_{2n})) := \Pi_{\cF_n}((\tilde{f}(X_{n+1}),\ldots,\tilde{f}(X_{2n}))
    % $
    % satisfies 
    \[
    \E \int_{\Omega} (\widehat{f}_1-f^{*})^2 d\PP_{k(n)}'  \leq \E \int_{\Omega} (\tilde{f}-f^{*})^2 d\PP_{k(n)}' = \E \int_{\Omega} (\tilde{f} - f^{*})^2 d\PP,
    \]
    since $\int (\cdot )^2 d\PP_{k(n)}'$ is just $\|\cdot \|^2/k(n)$. In order to conclude the minimax optimality of $\widehat{f}_1$, we know by Lemma \ref{Lem:Bro} that for any function in
    \[ 
      \cO:= \left\{f \in \cF_1:  \int_{\Omega}(f-\tilde{f}')^2 d\PP_{k(n)}' = 0\right\},
    \]
    % \[ (f(X_{n+1}),\ldots,f(X_{2n})) = (\tilde{f}'(X_{n+1}),\ldots,\tilde{f}'(X_{2n})) \]
    it holds that
    \[
        \E \int (f-f^{*})^2 d\PP  \leq 2\E \int(\tilde{f}-f^{*})^2  d\PP_{k(n)}' +  Ck(n)^{-\frac{4}{d}} \leq 2\E \int(\tilde{f}-f^{*})^2  d\PP + C_1n^{-\frac{4}{d+4}},
     \]
    where we used  \eqref{Eq:FixedDesign} and the fact that $k(n) = n^{\frac{d}{d+4}}$. Since we showed that $\widehat{f}_1$ must lie in $\cO$, the minimax optimality of this proper estimator follows.
    % (under its unrealistic assumption).

\end{document}